\documentclass[11pt]{amsart}
\usepackage[utf8]{inputenc}

\usepackage{overpic, amssymb, times, graphicx, tikz-cd,float}
\usepackage[all]{xy}
\usepackage[backref=page]{hyperref}
\usepackage{verbatim}
\usepackage{cleveref}

\hypersetup{
    colorlinks=true,
    linkcolor=blue,
    filecolor=blue,      
    urlcolor=blue,
    citecolor=blue,
}
\numberwithin{equation}{section}

\DeclareMathOperator{\wind}{wind}
\DeclareMathOperator{\Aut}{Aut}

\DeclareMathOperator{\ind}{ind}

\DeclareMathOperator{\Id}{Id}

\newcommand{\C}{\mathbb{C}}
\newcommand{\R}{\mathbb{R}}
\newcommand{\Z}{\mathbb{Z}}
\newcommand{\N}{\mathbb{N}}
\newcommand{\Q}{\mathbb{Q}}
\newcommand{\disk}{\mathbb{D}}

\newcommand{\be}{\begin{enumerate}}
\newcommand{\ee}{\end{enumerate}}

\newcommand{\Sphere}{\mathbb{S}}

\newtheorem{thm}{Theorem}[section]
\newtheorem{theorem}[thm]{Theorem}

\newtheorem{ex}[thm]{Example}

\newtheorem{proposition}[thm]{Proposition}

\newtheorem{defn}[thm]{Definition}

\newtheorem{lemma}[thm]{Lemma}
\newtheorem{cor}[thm]{Corollary}

\newtheorem{question}[thm]{Question}
\newtheorem{remark}[thm]{Remark}

\def\virdim{{\rm{virdim}}}
\def\SS{{\mathbb{S}}}
\def\rd{{\rm d}}

\def\CP{{\mathbb{CP}}}

\def\cM{{\mathcal M}}
\def\cN{{\mathcal N}}
\def\cO{{\mathcal O}}

\def\std{{\mathrm{std}}}
\def\ima{{\mathrm{im}}}
\def\wind{{\mathrm{wind}}}
\def\Hess{{\mathrm{Hess}}}
\def\Bd{{\mathrm{Bd}}}
\def\LCH{{\mathrm{LCH}}}

\newcommand{\Addresses}{{
		\bigskip
		\footnotesize

	    Zhengyi Zhou, \par\nopagebreak
        \textsc{State Key Laboratory of Mathematical Sciences, Chinese Academy of Sciences;}\par\nopagebreak
	    \textsc{Morningside Center of Mathematics, Chinese Academy of Sciences;}\par\nopagebreak
         \textsc{Academy of Mathematics and Systems Science, Chinese Academy of Sciences, China}\par\nopagebreak
		\textit{E-mail address}: \href{mailto:zhyzhou@amss.ac.cn}{zhyzhou@amss.ac.cn}

}}
\title{Unknottedness of symplectic submanifold fillings}
\author{Zhengyi Zhou}
\date{\today}
\begin{document}
	\maketitle
\begin{abstract}
We show that any symplectic filling of the standard contact submanifold $(\Sphere^{2n-1},\xi_{\std})$ of $(\Sphere^{2n+1},\xi_{\std})$ in $(\disk^{n+1},\omega_{\std})$ is smoothly unknotted if $n\ge 2$. We also give a self-contained proof of the Siefring intersection formula between punctured holomorphic curves and holomorphic hypersurfaces used in the proof using the $L$-simple setup of Bao-Honda. 
\end{abstract}
\section{Introduction}
Rigidity is one of the central themes in symplectic topology. Among all the rigidity phenomena, the rigidity of fillings is one of the most studied rigidity starting from Gromov's seminal work \cite{Gromov}. Filling rigidity results can often be transformed into obstruction results, yielding many interesting applications in symplectic topology \cite{EliashbergT3,Luttinger,Nemirovski}. One of the landmark results in the direction of filling rigidity is the Eliashberg-Floer-McDuff theorem \cite{McDuff91}, asserting that any exact filling of the standard contact sphere of dimension at least $3$ is diffeomorphic to a ball, generalizing Gromov's result in dimension $3$ to general dimensions in the smooth category. Such results have been generalized in various directions in the past 30 years \cite{OV,BGZ,Z23}. The rigidity of the filling is also studied in the relative setting, although most in dimension $3$. For example, Eliashberg-Polterovich \cite{EP}'s result on Lagrangian fillings of the standard Legendrian unknot. In this paper, we consider the relative version of filling rigidity for contact submanifolds. More precisely, we consider the most simple contact submanifold in the most simple contact manifold, namely the standard $(\Sphere^{2k-1},\xi_{\std})$ in $(\Sphere^{2n+1},\xi_{\std})$ for $k\le n$, where the contact submanifold $(\Sphere^{2k-1},\xi_{\std})$ is given by the intersection of $(\Sphere^{2n+1},\xi_{\std})\subset \C^{n+1}$ with complex linear subspaces. We consider symplectic fillings of $(\Sphere^{2k-1},\xi_{\std})$ inside the symplectic ball filling of $(\Sphere^{2n+1},\xi_{\std})$. We say the filling is symplectically unknotted if the completed filling pair is symplectomorphic to $(\C^{n+1},\C^k,\omega_{\std})$ and is smoothly unknotted if the completed filling pair is diffeomorphic to $(\C^{n+1},\C^k)$. When the codimension is at least $4$, as the $h$-principle governs the problem \cite[\S 20.1.1]{zbMATH07792729}, smooth knots in \cite{knots} of high codimension can be turned into symplectic knots (after deleting a point) which are smoothly nontrivial. Therefore, it only makes sense to study the rigidity phenomena in the codimension $2$ case.

In dimension $4$, the strongest rigidity results can be obtained, e.g.\ Gromov \cite{Gromov} showed that any exact filling of $(\Sphere^3,\xi_{\std})$ is the symplectically standard ball up to deformation, Eliashberg-Polterovich \cite{EP} showed that any exact Lagrangian filling of the standard Legendrian unknot has to be the standard Lagrangian disk. In the contact submanifold case considered here, using adjunction, we can argue that any symplectic filling of the standard transverse unknot must be a disk, and then the proof of Gromov's theorem implies that any such filling must be symplectically unknotted.  

In recent years, fundamental phenomena have been discovered concerning the geography of codimension $2$ contact submanifolds due to the work of Casals-Etnyre \cite{zbMATH07184225}, C{\^o}t{\'e}-Fauteux-Chapleau \cite{zbMATH07821993}, and Avdek \cite{avdek2023stabilization}. In the 2024 AIM workshop on higher-dimensional contact topology, Casals raised the following question regarding the filling of codimension $2$ contact submanifolds:
\begin{question}[Casals]
	Are symplectic fillings of the standard $(\Sphere^3,\xi_{\std})$ in $(\Sphere^{5},\xi_{\std})$ symplectically unknotted\footnote{Here is the \href{http://aimpl.org/highdimcontacttop/9/}{link} to the question.}?
\end{question}
The main purpose of this paper is to explain the first and easier step toward the above question, that is, those codimension $2$ symplectic fillings are necessarily smoothly unknotted in any dimension. This already exhibits rigidity in the relative setting, as there are abundant knotted smooth balls of codimension $2$. 

\begin{theorem}\label{thm:main}
    Let $W$ be a symplectic filling of the standard contact submanifold $(\Sphere^{2n-1},\xi_{\std})$ of $(\Sphere^{2n+1},\xi_{\std})$ in $(\disk^{n+1},\omega_{\std})$ for $n\ge 2$, then $(\disk^{n+1},W)$ is smoothly unknotted, where $\disk^{n+1}\subset \C^{n+1}$ is a unit disk.
\end{theorem}
\begin{remark}
By the celebrated Eliashberg-Floer-McDuff theorem \cite{McDuff91}, $W$ must be diffeomorphic to $\disk^{n}$.  Moreover, we can replace the standard filling $(\disk^{n+1},\omega_{\std})$ in \Cref{thm:main} by any exact filling of $(\Sphere^{2n+1},\xi_{\std})$, the arguments go through without any change. 
\end{remark}
In general, we can establish the following relative filling rigidity for contact manifold pairs.
\begin{theorem}\label{thm:V}
    Let $V$ be a Liouville domain and $U\subset V$ a codimension $2$ symplectic submanifold. Then for any symplectic filling $W$ of $\partial(U\times \disk)\subset \partial (V\times \disk)$ in $V\times \disk$, we have $\pi_1(V \times \{(0,1)\} \backslash U \times \{(0,1)\})\to \pi_1(V\times \disk\backslash W)$ is surjective. If $\pi_1(V\backslash U)$ is abelian, then $V\times \disk\backslash W$ is homotopy equivalent to $V\backslash U$
\end{theorem}

The contact manifolds in \Cref{thm:V} have strong uniqueness results for their Liouville fillings \cite{BGZ,OV,Z23}. On the other hand, we have many examples of contact manifolds in all dimensions with infinitely many non-diffeomorphic Weinstein fillings \cite{zbMATH02041154, zbMATH07949277}. In fact, examples from \cite{zbMATH02041154, zbMATH07949277} are differed by their homology groups. The following rigidity result implies that many fillings can not be realized as a submanifold filling.  

\begin{theorem}\label{thm:binding}
    Let $W$ be a symplectic filling of $\partial V \times \{0\} \subset \partial (V\times \disk)$ in $V\times \disk$. Then $W\subset V\times \disk$ induces an isomorphism on homology and is surjective on fundamental groups. In particular, if $\pi_1(V)$ is abelian, we have $W\subset V\times \disk$ is a homotopy equivalence. 
\end{theorem}

The proofs of \Cref{thm:main} and \Cref{thm:V} are based on pseudo-holomorphic curves as in the absolute filling rigidity case. We can view $W$ as a holomorphic hypersurface by choosing appropriate almost complex structures. With the help of Siefring's intersection theory, we can detect the topology of the submanifold filling complement using pseudo-holomorphic curves. The proof of \Cref{thm:binding} is also based on intersection theory, but is applied differently. Siefring's intersection theory was used by C{\^o}t{\'e} and Fauteux-Chapleau \cite{zbMATH07821993} to define contact isotopic invariants of codimension $2$ submanifolds from SFT. However, the study of (relative) topology of symplectic submanifolds based on Siefring's intersection theory seems to be new.

Although we expect the above rigidity results to be special for the contact pairs in \Cref{thm:V} and \Cref{thm:binding}, we shall not have such a strong rigidity for a generic contact pair. It is not easy to construct contact manifolds with multiple fillings.  Moreover, there are no known examples of contact pairs with filling pairs that are smoothly the same but symplectically different. 
Another natural question is how to convert rigidity results in \Cref{thm:main} into obstruction results as in \cite{EliashbergT3} producing other applications in symplectic and contact topology, especially in understanding contact submanifolds in $(\Sphere^{2n+1},\xi_{\std})$.

We will review Siefring's intersection theory in \S \ref{s2} and give a self-contained proof of the intersection formula promised in \cite{moreno2019holomorphic}. The intersection formula is of a more restrictive form compared to that of \cite{moreno2019holomorphic} but is sufficient for applications in this paper and other topological applications like computation of the contact homology in \cite{avdek2024bourgeois} as well as the contact homology invariants with intersection information in \cite{zbMATH07821993}\footnote{Strictly speaking, one needs to modify \cite{zbMATH07821993} using almost complex structures in \cite{BH}, this works fine with Pardon's VFC theory \cite{Pardon} used in \cite{zbMATH07821993}}. We explain some of the almost complex structures in \S \ref{s3}, in particular, we build a holomorphic foliation on $\widehat{V\times \C}$ that is compatible with the contact structure to control homology classes of curves using classical intersection theory. Such a complication is necessary, as we assume neither $H_2(V;\Q)=0$ nor $c_1(V)=0$, homology classes are important to determine virtual dimensions as well as intersection numbers for later applications of Siefring's intersection theory. We will explain the holomorphic curves used to probe the knot complement in \S \ref{s4} and prove \Cref{thm:main} and \Cref{thm:V} in \S \ref{s5}. \Cref{thm:binding} is proved in \S \ref{s6}.

\subsection*{Acknowledgments}
The author is grateful to members of the symplectic fillings for the contact submanifolds group in the problem session of the 2024 AIM workshop on higher-dimensional contact topology for helpful discussions. The author is supported by the National Key R\&D Program of China under Grant No.\ 2023YFA1010500, the National Natural Science Foundation of China under Grant No.\ 12288201 and 12231010.
\section{Siefring's  intersection theory}\label{s2}
In \cite{Siefring11}, Siefring developed the intersection theory for punctured holomorphic curves in dimension $4$. The cornerstone of the intersection theory is the asymptotic expansion of a holomorphic curve near the punctures, whose $4$-dimensional case was established by Hofer-Wysocki-Zehnder \cite{HWZ1,HWZ2} and Siefring \cite{Siefring08}. In higher dimensions, such formulas are harder to derive for a general setup. A version of the higher-dimensional analog was announced by Siefring \cite{Sie} and was summarized in \cite[Theorem 2.2]{moreno2019holomorphic}. On the other hand, the $L$-simple formulation in the works of Bao-Honda \cite{BH-cylindrical,BH} makes the Cauchy-Riemann equation linear near a puncture, hence one can get an asymptotic formula to derive the intersection theory. In this section, we will explain the intersection theory in the $L$-simple setup. Although the $L$-simple set-up uses special contact forms as well as special almost complex structures near embedded Reeb orbits, it causes no loss of generality in applications because, in higher-dimensional questions, holomorphic hypersurfaces are often constructed by hand where those $L$-simple restrictions can usually be arranged. 
\subsection{Simple neighborhood of Reeb orbits}\label{ss:L-simple}
The following classes of almost complex structures are used to define symplectic field theory invariants of contact and Liouville domains in \cite{BH-cylindrical,BH}.
\begin{defn}\label{def:tame}
Let $\alpha$ be a contact form of a contact manifold $(Y^{2n-1}, \xi)$ with Reeb field $R$ and let $s$ be a coordinate on $I\subset \R$.
An almost complex structure $J$ on an $I$-symplectization $I\times Y$ is \emph{$\alpha$-tame} if
\begin{enumerate}
    \item $J$ is invariant under translation in the $s$-coordinate,
    \item $J(\partial_s)=F_J R$ for some $F_J\in C^\infty(Y,(0,\infty))$,
    \item there is a $2n-2$-plane field $\xi_J \subset TY$ satisfying $J\xi_J=\xi_J$, and
    \item $\rd \alpha(V, JV)>0$ for all non-zero $V\in \xi_J$.
\end{enumerate}
\end{defn}
In general, $\alpha$ will be nontrivial on $\xi_J$ so that this hyperplane field will not necessarily agree with $\xi$. Nevertheless, $J$ will be tamed by $\rd (e^{\delta s }\alpha)$ for $\delta \ll 1$, or equivalently it will be tamed by $\rd(e^s\alpha)$ for $F_J\gg 0$. In practice, the almost complex structure will almost be a usual $\alpha$-admissible almost complex structure, i.e.\ $F_J=1$ and $\xi_J$ is close to $\xi$ and $(\rd \alpha, J)$ is compatible on $\xi_J$. In this case, it will be tamed by $\rd(e^s\alpha)$. Compared to the usual $\alpha$-compatible almost complex structures \cite[\S 1.4]{zbMATH01643843}, tameness alone does not provide enough analytical foundations to define moduli spaces of holomorphic curves with punctures, see \Cref{rmk:ana} for extra ingredients needed.

Following the $L$-simple formulation \cite[Definition 3.1.1]{BH}, for any $L>0$, there exists a $C^0$-small perturbation $\alpha_Y$ of the contact form, such that
for any simple Reeb orbit $\gamma\subset Y$ of period at most $L$, there is a neighborhood $N_Y(\gamma)$ of $\gamma$ modeled on $[0,T:=\int \gamma^*\alpha_Y]_t\times \disk^{n-1}_{\delta}/\sim$, where $\disk^{n-1}_{\delta}$ is a radius $\delta$ ball in $\C^{n-1}$. Here, the equivalence relation $\sim$ is multiplication by $-1$ on some of the coordinates of $\C^{n-1}$ depending on the linear return map of $\gamma$ \cite[Lemma 3.1.2]{BH}. For our purpose here, it causes no loss of generality if we only consider the identity map on $\disk^{n-1}$ to glue, i.e.\ $[0,T]_t\times \disk^{n-1}_{\delta}/\sim=\Sphere^1_t \times \disk^{n-1}_{\delta}$. The contact form $\alpha_Y$ over such a neighborhood is given by 
$$(1+Q)\rd t+\beta$$
where $Q$ is a non-degenerate pure quadratic function (i.e.\ without constant term and linear terms) on $\disk^{n-1}_{\delta}$ and $\beta=\frac{1}{2}\sum_{i=1}^{n-1}(x_i\rd y_i-y_i\rd x_i)$. The Reeb vector field is given by $(1+Q-\beta(X_Q))(\partial_t-X_{Q})$,  where $X_Q$ is the Hamiltonian vector field of $Q$ w.r.t.\ $\rd \beta$, i.e.\ $\iota_{X_Q}\rd \beta=-\rd Q$. A direct computation shows that $Q=\beta(X_Q)$ and hence the Reeb vector field $R_Y$ is $\partial_t-X_Q$.

Following \cite[Definition 3.1.4]{BH}, an $L$-simple almost complex $J$ on $\widehat{Y}$ satisfies 
\begin{enumerate}
    \item $J$ is the standard complex structure on $\disk^{n-1}_{\delta}$ and $\xi_J$ is $T\disk^{n-1}_{\delta}$.
    \item $J(\partial_s)=R_Y=\partial_t-X_Q$\footnote{In \cite[Definition 3.1.4]{BH}, the authors wrote $J\partial_s=gR_Y$ for $g=1+Q-\beta(X_Q)$ in the a neighborhood of $\gamma$ and $g=1$ outside a neighborhood of $\gamma$. As $1+Q-\beta(X_Q)=1$, we can set $J\partial_s=R_Y$ globally in \cite{BH}.}.
    \item $J$ is $\alpha_Y$-tame.
\end{enumerate}
Note that $\xi_J=\xi$ along $\gamma$, we can arrange that $\xi_J=\xi$ outside of $N_Y(\gamma)$ for all of those simple Reeb orbits with period at most $L$ and $\xi_J$ is $C^\infty$ close to $\xi$ inside those neighborhoods. By imposing the condition that $(\rd \alpha, J)$ is compatible on the symplectic subspace $\xi_J$, our $J$ is $C^\infty$ close to an $\alpha$-compatible almost complex structure in the usual sense. In particular, $e^s$ is strictly plurisubharmonic w.r.t.\ $J$.

As $Q$ is a pure quadratic function, there exists $2n-2$ by $2n-2$ constant matrix $S_Q$, such that $X_Q$ in basis $(\partial_{x_1},\ldots, \partial_{x_{n-1}},\partial_{y_1},\ldots, \partial_{y_{n-1}})$ is given by $S_Q \cdot (x_1,\ldots,x_{n-1},y_1.\ldots, y_{n-1})^T$. Then, with suitable cylindrical coordinates near the puncture, the Cauchy-Riemann equation for $u$ near a puncture asymptotic to $\gamma^k$ can be written as the linear equation:
\begin{equation}\label{eqn:linearCR}
    \frac{\partial \eta}{\partial s}+J\frac{\partial \eta}{\partial t}+JS_Q\eta=0
\end{equation}
for $u(s,t)=(s,t,\eta(s,t))$ with $\eta(s,t)\in \disk^{n-1}_{\delta}$ and $(s,t)\in \R_{\pm}\times [0,kT]/0\sim kT$ with complex structure $j\partial_s=\partial_t$, see \cite[(3.1.2)]{BH}. It is straightforward to compute that $S_Q=J\Hess_Q$.

\begin{remark}\label{rmk:ana}
    Here we remark that using such almost complex structures causes no analytic issue for holomorphic curves with finite Hofer energy:
    \begin{enumerate}
        \item The relevant SFT compactness theorem for $\alpha$-tame almost structures can be found in  \cite[§3.4]{BH-cylindrical}.
        \item The asymptotic analysis for holomorphic curves using $L$-simple almost complex structures holds as the equation is linear in a standard form near the punctures. The essence is that the asymptotic operator is self-adjoint w.r.t.\ a suitable $L^2$-norm. The $L$-simple case is a special case of \cite[\S 7.2]{wendl2016lectures}, as $(\rd t,\alpha)$ gives a stable Hamiltonian structure near the Reeb orbits, and an $L$-simple almost complex structure is compatible with this stable Hamiltonian structure. For general $\alpha$-tame almost complex structures, the self-adjointness of the total asymptotic operator is needed, as all literature on the asymptotic behavior and Fredholm properties is built upon the self-adjoint asymptotic operators. For $L$-simple almost complex structures, the asymptotic operator is self-adjoint near non-constant orbits of Hamiltonians that only depend on $s$. Hence, it is also possible to define symplectic cohomology for such Hamiltonians and $L$-simple almost complex structures. 
        \item  For the purpose of defining symplectic cohomology, we also need $e^s$ to be plurisubharmonic to ensure a maximum principle to get $C^0$ bound. This holds as long as $\xi_J$ is close to $\xi$, for example.
    \end{enumerate}
\end{remark}

\subsection{Siefring's intersection theory}\label{ss:intersection}
In this section, we review Siefring's intersection theory using the $L$-simple setup. The chain of arguments is identical to the $4$-dimensioncal case, for which we will follow Wendl's book \cite{Wendlbook}. Let $W$ be a symplectic (strong) cobordism from $Y_-$ to $Y_+$. Now we consider codimension $2$ contact submanifolds $H_{\pm}\subset Y_{\pm}$ with trivial normal bundles and a codimension $2$ symplectic hypersurface $H\subset W$ with $H\pitchfork Y_{\pm}=H_{\pm}$. We assume the contact form on a neighborhood of $H_{\pm}$ is given by $(H_{\pm}\times \disk_{\delta}, (1+f_{\pm})\alpha_{H_{\pm}}+x\rd y-y\rd x)$, where $\alpha_{H_{\pm}}$ is an $L$-simple contact form on $H_{\pm}$ and $f_{\pm}:\disk:\to \R$ are non-degenerate pure quadratic functions. Although different choices of $f$ will yield homotopic contact structures on $H_{\pm}\times \disk_{\delta}$ (not relative to the boundary), the specific choices of $f$ i.e.\ whether it is a saddle point or a local minimum/maximum as well as the size of the Hessian, will be enforced by the global contact topology in applications. 

We can compute that the Reeb vector in those neighborhoods is given by 
$$R_{Y_{\pm}}=R_{H_{\pm}}-X_f=\partial_t-X_Q-X_f.$$
By choosing an $L$-simple almost complex structure $J_{H_\pm}$ on $H_{\pm}$, we have an almost complex structure $J_{\pm}$ on $Y_{\pm}$ defined by $J_{\pm}\partial_s=R_{Y_{\pm}}$ and $J=J_0$, the standard complex structure, on $\disk^{n-1}_{\delta}\oplus \disk^2_{\delta}$ near the neighborhood of Reeb orbits of $H_{\pm}$ and $\alpha_{Y_{\pm}}$-compatible outside the neighborhoods (with an interpolation domain in between). In particular, $\widehat{H}_{\pm}$ is a holomorphic hypersurface in $\widehat{Y}_{\pm}$. We equip $\widehat{W}$ with a tame almost complex structure that is the above almost complex structures near the positive/negative ends of $W$. Moreover, we can arrange $\widehat{H}$ to be a holomorphic hypersurface in $\widehat{W}$. Then 
the Cauchy-Riemann equation for $u$ near a puncture asymptotic to $\gamma^k\subset H_{\pm}\times \disk_{\delta}$ can be written as the linear equation:
\begin{eqnarray} 
    \frac{\partial \eta}{\partial s}+J_0\frac{\partial \eta}{\partial t}-\Hess_Q\eta & = & 0 \nonumber \\
    \frac{\partial \eta_{\perp}}{\partial s}+J_0\frac{\partial \eta_{\perp}}{\partial t}-\Hess_f\eta_{\perp} & = & 0 \label{eqn:linearCR_perp}
\end{eqnarray}
for $u(s,t)=(s,t,\eta(s,t),\eta_{\perp}(s,t))$ with $\eta(s,t)\in \disk^{n-1}_{\delta}, \eta_{\perp}(s,t)\in \disk_{\delta}$ and $(s,t)\in \R_{\pm}\times [0,kT]/0\sim kT$ with complex structure $j\partial_s=\partial_t$. We write $A_f=-J\frac{\rd}{\rd t}+\Hess_f$, which is a self-adjoint operator. The non-degeneracy assumption implies that $0$ is not an eigenvalue of $A_f$. We use $\ldots<a_{-1}<0<a_1<a_2<\ldots$ to denote the spectrum $\sigma(A_f)$ of $A_f$. Then $\eta_{\perp}$ has the following expansion depending on whether $u$ is asymptotic to $\gamma^k$ positively or negatively. 
\begin{equation}\label{eqn:asym}
    \eta_{\perp} = \sum_{i=1}^{\infty} e^{a_is}\eta_i(t), s\in \R_- \text{ or }  \eta_{\perp} = \sum_{i=-1}^{-\infty} e^{a_is}\eta_i(t), s\in \R_+
\end{equation}
where $\eta_i(t)$ is an eigenfunction of $A_f$ with eigenvalue $a_i$. It is clear that each $\eta_i(t)$ is nowhere zero, therefore, we have a well-defined winding number $\wind(\eta_i):= \deg (\eta_i(t)/|\eta_i(t)|)$.
\begin{proposition}[{\cite[Theorem 3.15]{Wendlbook}}]\label{prop:wind}
There exists a well-defined integer-valued function
$$\wind:\sigma(A_f)\to \Z$$
defined by $\wind(a_i)=\wind(\eta_i)$. This map is increasing monotonically and attains each value in $\Z$ exactly twice (counting the multiplicity of eigenvalues).
\end{proposition}

Now, given a holomorphic curve $u:\Sigma\to \widehat{X}$, such that the asymptotic orbits of $u$ are either Reeb orbits of the period at most $L$ contained in $H_{\pm}$, or Reeb orbits do not intersect $H_{\pm}$. 

\begin{proposition}
    If $\ima u\not \subset \widehat{H}$, then there is a finite number of points $p$ in $\Sigma$, such that $u(p)\in \widehat{H}$.
\end{proposition}
\begin{proof}
    By the unique continuation, the normal asymptotic expansion of $\eta_{\perp}$ \eqref{eqn:asym} must not be zero, hence there are no intersection points near the punctures. Then the claim follows from \cite[Exercise 2.6.1. (i)]{McDuff91}.
\end{proof}
For each $p\in \Sigma$, such that $u(p)\in \widehat{H}$, we can define the local intersection number $\delta(p,u,H)$ be the total intersection number of $u$ and $\widehat{H}$ around $p$ by a generic perturbation. Such a number is strictly positive by the positivity of intersection and is one if and only if they intersect transversely \cite[Exercise 2.6.1]{McDSal}.

We define $u_{\tau}$ to be the push-off of $u$ near the punctures by the following rule: we push off $u$ in the direction of $\eta_{-1}$ near positive punctures and in the direction of $\eta_{1}$ near negative punctures. After push-off, $u_{\tau}$ is not asymptotic to $\widehat{H}$, in particular, they have well-defined topological intersection numbers $u_{\tau}\cdot \widehat{H}$ yielding the following definition.
\begin{defn} 
   We define the  intersection number
    $u\bullet H := u_{\tau}\cdot \widehat{H}$.
\end{defn}
$u\bullet H$ only depends on the homotopy class of $u$ as maps that are convergent to those Reeb orbits, see \cite[\S 4.2]{Wendlbook}.                         

Now given a such holomorphic curve $u$, for a positive puncture $p$ at which $u$ is asymptomatic to an orbit in $H$, we define the hidden intersection at $p$ to be 
$$\delta_\infty(p,u,H)=\wind(\eta_{-1})-\wind_p(u)$$
where $\wind_p(u)$ is defined to be the winding number of the leading term eigenfunction of $u$ in the asymptotic expansion near $p$. Similarly, for a negative puncture $q$, we define the hidden intersection at $q$ to be 
$$\delta_{\infty}(q,u,H)=\wind_q(u)-\wind(\eta_{1}).$$
For the remaining punctures, we define $\delta_{\infty}(p,u,H)=0$. By \Cref{prop:wind} and \eqref{eqn:asym}, we have $\delta_\infty(u,p,H)\ge 0$ for a holomorphic $u$. By \eqref{eqn:asym}, the intersection number in $u_{\tau} \cdot \widehat{H}$ near a puncture is measured by the discrepancy of the winding numbers, i.e.\ $\delta_\infty(p,u,H)$. Hence, we have the following formula. 
\begin{theorem}\label{thm:intersection}
Let $u$ be a holomorphic curve with finite Hofer energy. Assume that all asymptotic orbits of $u$ that are contained in $\widehat{H}$ have periods at most $L$.  If $\ima u\not \subset \widehat{H}$, then we have 
$$u\bullet H = \sum_{p\in \Sigma, u(p)\in \widehat{H}}\delta(p,u,H)+\sum_{p\in \Gamma^+\cup \Gamma^-} \delta_{\infty}(p,u,H),$$
where $\delta(p,u,H)>0$ and $\delta_{\infty}(p,u,H)\ge 0$.
\end{theorem}

\begin{cor}\label{cor:leaf}
    If $u\bullet H = 0$, then either $\ima u\subset \widehat{H}$ or $\ima u \cap \widehat{H} =  \emptyset$. If $u\bullet H <0$, then we must have  $\ima u\subset \widehat{H}$.
\end{cor}

\begin{remark}
As long as the Cauchy-Riemann equation can be decoupled into the hypersurface direction and the normal direction, and the equation in the normal direction has at least an asymptotic expansion in the form of \cite[Theorem 3.11]{Wendlbook}, e.g.\ the equation is linear as in \eqref{eqn:linearCR_perp}, we will have the intersection theory. The form of the equation in the hypersurface direction, e.g. \eqref{eqn:linearCR}, does not affect the theory. Therefore, one can develop an intersection theory for contact forms and almost complex structures that are only ``$L$-simple" in the normal direction. This is developed and used in \cite{DG}.
\end{remark}

\begin{remark}
   The trivial normal bundle assumption of $H_{\pm}\subset Y_{\pm}$ is just for the cleanness of the statement, since we have a global $L$-simple neighborhood of $H_{\pm}$. In applications,  we only need $L$-simple models near Reeb orbits. Since over a neighborhood of a simple Reeb orbit in $H_{\pm}$, the normal bundle must be trivial.
\end{remark}

\begin{ex}\label{ex:eigen}
    Throughout this paper, we will only encounter \eqref{eqn:linearCR_perp} such that $f$ is a local maximum with a small Hessian, i.e.\ $f=-\epsilon(x^2+y^2)$. In this case, $A_f=-J\frac{\rd}{\rd t}-2\epsilon$. As a consequence. we have $a_{-1}=-2\epsilon$ with multiplicity $2$. The corresponding eigenfunctions are constant functions on $S^1$. As a consequence. $\wind(a_{-1})=0$.
\end{ex}

\subsection{Foliations of holomorphic hypersurfaces }
Although this will not be used here, intersection theory can help confine holomorphic curves in a leaf in the presence of foliations by holomorphic hypersurfaces as outlined in \cite{moreno2019holomorphic}. To incorporate the special form of intersection theory in \Cref{thm:intersection}, we need to use special foliations, which are sufficient for applications in \cite{avdek2023algebraic,avdek2024bourgeois,DG}.

Let $H$ be a codimension $2$ contact submanifold with a trivial normal bundle in $Y$. We assume the contact form near $H$ is $L$-simple in \S \ref{ss:L-simple}, i.e.\ around a simple Reeb orbit $\gamma$, we have
$$\alpha = (1+Q+f)\rd t+\frac{1}{2}\sum_{i=1}^{n-1}(x_i\rd y_i-y_i\rd x_i)+\frac{1}{2}(x_n\rd y_n-y_n\rd x_n)$$
with $Q$ is a pure quadratic function in $\{x_i,y_i\}_{1\le i \le n-1}$ and $f$ is a pure quadratic function in $x_n,y_n$. Given an $L$-simple almost complex structure near $\gamma$, we have local $J$-holomorphic foliations as follows: Let $\nu:\R_{\pm}\to \disk_{\delta}$ be the gradient flow of $f$ w.r.t.\ the Euclidean metric. By our choice of $J$, we have 
$$\R_{\pm}\times \Sphere^1_t\times \disk^{n-1}_{\delta} \to \R_{\pm}\times \Sphere^1_t\times \disk^{n-1}_{\delta} \times \disk_{\delta}, \qquad (s,t,z)\mapsto (s,t,z, \nu(s))$$
is holomorphic w.r.t.\ the almost complex structure on  $\R_{\pm}\times \Sphere_t\times \disk^{n-1}_{\delta}\subset \widehat{H}$ induced from $\widehat{H}$. Those foliations project to gradient flow lines in $\disk_{\delta}$. 
\begin{defn}
    Let $W$ be a symplectic cobordism with completion $(\widehat{W},J)$. We say that a codimension $2$ foliation on $\widehat{W}$ is $L$-simple, if all positive/negative ends of holomorphic leaves are the local foliation given above near all simple Reeb orbits of periods at most $L$ that are contained in the asymptotic of the leaves. 
\end{defn}

\begin{theorem}\label{thm:intersection_foliation}
    \Cref{thm:intersection} holds for leaves in an $L$-simple foliation.
\end{theorem}
The meaning of the hidden intersection at infinity in \Cref{thm:intersection_foliation} will be clear from the proof below.
\begin{proof}[Proof of \Cref{thm:intersection_foliation}]
    Note that $\nu(s) = e^{as}\nu_0$, where $a$ is an eigenvalue of $\Hess_f$ and $\nu_0$ is an eigenvector. In particular, $\nu(s)$ is also in the form of the asymptotic expansion \eqref{eqn:asym}. The hidden intersection of $u$ with the leaf determined by $\nu(s)$ can be computed using the winding number of the leading term of $\eta_{\perp}-\nu$, where $\eta_{\perp}$ is the normal component of $u$ near punctures that are asymptotic to Reeb obits in the leaves. As $\eta_{\perp}-\nu$ has also the form \eqref{eqn:asym}, the proof for \Cref{thm:intersection} goes through.
\end{proof}

\begin{cor}
    Assuming there is a holomorphic hypersurface foliation as above and $u$ is a holomorphic curve such that all asymptotic Reeb orbits that are asymptotic to leaves have periods at most $L$, if $u\bullet F\le 0$ for any leaf $F$, then $u$ is contained in a leaf. 
\end{cor}
\begin{proof}
    If $u$ is not contained in any leaf, then $u$ must intersect one leaf $F$ and is not contained in $F$. By \Cref{thm:intersection_foliation}, we have $u\bullet F\ge 1$, contradicting the assumption.
\end{proof}
As $u\bullet F$ is topological, oftentimes, one can argue that $u\bullet F$ is independent of the leaf or only depends on finitely many different types of leaves. Therefore, the condition that $u\bullet F\le 0$ for any leaf $F$ is reasonable to check.

\section{Contact pairs and their $L$-simple setup}\label{s3}
\subsection{$L$-simple contact forms}\label{ss:L-simple-contact}
We describe the contact pairs that we will consider in this paper and present them in an $L$-simple manner. Throughout this paper, the ambient contact manifold is always in the form of $\partial(V\times \disk)$, where $V$ is a Liouville domain. Let $\lambda$ denote a Liouville form on the Liouville domain $V$, such that the Reeb orbits on $\Gamma:=\partial V$ have very large period. We use $\alpha_{\Gamma}$ denote the restriction of $\lambda$ to $\Gamma$ and $R_{\Gamma}$ is the Reeb vector field. We use $r$ to denote the collar coordinate on $\Gamma$, such that the completed Liouville manifold $(\widehat{V},\widehat{\lambda})$ is given by $V\cup \Gamma \times (1,\infty)_r$ with $\widehat{\lambda}=\lambda$ on $V$ and $\widehat{\lambda}=r\alpha_{\Gamma}$ on $\Gamma \times (1,\infty)_r$. Note that by following the negative flow of the Liouville vector field $r\partial_r$, the $r$ coordinate continues to exist in the interior of $V$ for $r\in (0,1)$. We first fix a Morse function $f$ on $V$, such that
\begin{itemize}
    \item $f$ is $C^2$ close to $1$;
    \item $\frac{\partial f}{\partial r}<0$ near $r=1$.
\end{itemize}
We can extend $f$ outside $V$ to $\Gamma \times (1,1+\epsilon)$ such that 
\begin{itemize}
    \item $f$ only depends on $r$ and  $\frac{\partial f}{\partial r}<0$;
    \item $f^{-1}$ can be extended over $\displaystyle \lim_{t\to 1+\epsilon}f(t)$ by the constant $1+\epsilon$ smoothly;
    \item $\lim_{r\to 1+\epsilon}f(r)=1$.
\end{itemize}
The hypersurface $Y_f$ in $(\widehat{V}\times \C, \widehat{\lambda}+\rho^2\rd \theta)$ given by the union of $\rho^2=f(x)$ and $r=1+\epsilon$, where $\rho$ is the Euclidean radius on $\C$. By our construction, $Y_f\simeq \partial(V\times \disk)$ is a smooth hypersurface with an induced contact form $\alpha$ from the restriction of $\widehat{\lambda}+\rho^2\rd \theta$. We use $Y_V$ to denote the $V\times \Sphere^1$ part, which is the graph of $\rho^2=f$ on $V$, with contact form $f\rd \theta + \lambda$ and $Y_\disk$ the $\Gamma \times \{1+\epsilon\} \times \disk$ with contact form $(1+\epsilon)\alpha_{\Gamma}+\rho^2\rd \theta$. And $Y_\cap$ is the complement of $Y_V,Y_{\disk}$, i.e.\ the graph of $\rho^2=f$ on $(1,1+\epsilon)_r$, with contact from $r\alpha_{\Gamma}+f\rd\theta$.

The Reeb vector field over $Y_V$ and $Y_\cap$, under the parametrization from $V\times \Sphere^1$ and $(1,1+\epsilon)\times \Gamma \times \Sphere^1$ (instead of viewing as a subset in $\widehat{V}\times \C$), is given by 
$$R=(f-\widehat{\lambda}(X_f))^{-1}(\partial_{\theta}-X_f)$$
where $X_f$ is the Hamiltonian vector field for symplectic form $\rd \widehat{\lambda}$ with the convention $\iota_{X_f}\rd \widehat{\lambda}=-\rd f$. 

Let $p$ be a critical point of $f$ on $V$. We can assume that there is a neighborhood $N(p)\subset V$, such that there is a diffeomorphism $(N(p),\lambda)\to (B^{2n}_{\epsilon}(0)\subset \R^{2n}, \frac{1}{2}\sum (x_i\rd y_i-y_i\rd x_i))$ preserving the Liouville forms. We assume that $f$ using this coordinate is a constant plus a pure quadratic function with a small Hessian. Following \cite[\S 6]{ADCI}, as $f-1$ is $C^2$ small on $V$ and $\Gamma=\partial V$ has long Reeb orbits, all Reeb orbits on $Y_f$ are either
\begin{enumerate}
    \item are multiple covers of the $\Sphere^1$ fiber over a critical point $p$ of $f$ on $V$;
    \item have long period, the threshold depends on the $C^2$-smallness of $f-1$ and the minimal period of Reeb orbits on $\Gamma$.
\end{enumerate}
We use $\gamma_p$ to denote the simple Reeb orbit over the critical point $p$. We can assume that all Reeb orbits with period $<4\pi$ are those $\gamma_p$ orbits. Near the Reeb orbit $\gamma_p$, the contact form on $\gamma_p\times N(p)\simeq \Sphere^1\times B^{2n}_{\epsilon}(0)$ is $f\rd \theta + \frac{1}{2}\sum (x_i\rd y_i-y_i\rd x_i)$, which is precisely in the $L$-simple form after we reparameterize the $\theta$-coordinate.

We will consider two types of contact submanifolds in $\partial(V\times \disk)$. 
\begin{enumerate}
    \item  The binding of  $\partial(V\times \disk)$ viewed as the trivial open book with page $V$, i.e.\ $\Gamma \times \{0\} \subset \partial(V\times \disk)$.
    \item  Given a codimension $2$ symplectic submanifold $U\subset V$, $H:=\partial(U\times \disk)\subset Y=\partial (V\times \disk)$ is a contact submanifold. We can choose the Liouville form on $V$, such that the Liouville vector field is tangent to $U$ near the boundary.  We require that $f$, such that the critical points of $f|_U$ are also critical points of $f$ on $V$. Moreover, near each critical point $p$ of $f|_U$, $(V,U,\lambda,f)$ is locally isomorphic to $(B^{2n}_{\epsilon}(0), B^{2n-2}_{\epsilon}(0), \frac{1}{2}\sum (x_i\rd y_i-y_i\rd x_i), Q-\epsilon(x_n^2+y_n^2))$, where $Q$ is a quadratic function in $x_1,\ldots,x_{n-1},y_1,\ldots,y_{n-1}$ without linear terms and $\epsilon>0$. Similar to the $Y_f$ construction, using $f_U$, we have a contact hypersurface $H_{f|_U}\subset Y_f$. It is clear from the construction that near all simple Reeb orbits in $\Sphere^1\times U \subset H_{f|_U}$, the contact form on $(Y_f,H_{f|_U})$ is $L$-simple as in \Cref{ss:intersection}. In particular, we have the intersection theory if we use an $L$-simple almost complex structure.
\end{enumerate}

\subsection{A compatible almost complex near the binding}\label{ss:acs_binding}
We say an $s$-invariant almost complex structure $J$ on $\widehat{Y}=\R_s\times Y$ is \textbf{compatible} if $J(R)=-C\partial_s$ for $C>0$ and $J$ preserve the contact structure where $J$ is compatible with the symplectic form $\rd\alpha$. With a bit of misuse of terminology, we will call $\alpha$-tame almost complex structure that is $C^\infty$ close to a compatible one compatible as well, e.g.\ the almost complex structures described in \S \ref{ss:intersection}. A complex structure on $\widehat{W}$--the completion of a Liouville domain--is compatible if it is compatible on the positive end of $\widehat{W}$ as the end of the symplectization and is tamed by the symplectic form on $W$. In the following, we describe a compatible almost complex structure on $\widehat{Y}_{\disk}$ with nice properties that will help us confine the topology of holomorphic/Floer cylinders.

We fix a compatible almost complex structure $J_{\Gamma}$ on $\Gamma$. Note that the contact structure on $Y_{\disk}$ is given by
$$\xi_{\Gamma}\oplus \left\langle \partial_x+\frac{y}{1+\epsilon}R_{\Gamma}, \partial_y-\frac{x}{1+\epsilon}R_{\Gamma}\right\rangle$$
where $\xi_{\Gamma},R_{\Gamma}$ are the contact structure and Reeb vector field on $\Gamma$ from the contact form $\alpha_{\Gamma}$. The Reeb vector $R$ is given by $(1+\epsilon)^{-1}R_{\Gamma}$. We consider the following compatible almost complex structure $J$ on $\widehat{Y}_{\disk}$:
\begin{equation}\label{eqn:J_good}
    J|_{\xi_{\Gamma}}=J_{\Gamma}, \quad J\left(\partial_x+\frac{y}{1+\epsilon}R_{\Gamma}\right)=\partial_y-\frac{x}{1+\epsilon}R_{\Gamma}, \quad JR=-C\partial_s, C>0.
\end{equation}
We will call such an almost complex structure good on $Y_{\disk}$.

\begin{proposition}\label{prop:d}
    Let $\Sigma$ be a subdomain of $\R_s\times \SS^1_t$, $H$ a Hamiltonian on $\widehat{Y}_{\disk}$ depending only on the cylindrical coordinate $s$ and $J$ is good as above. Then any solution $u:\Sigma \to \widehat{Y}_{\disk}$ of $\partial_su+J(\partial_tu-X_H)=0$, we have 
    $$\int u^*\alpha\ge \int (\pi_{\disk}\circ u)^*(\rd(\rho^2\rd \theta))\ge 0$$
    and $\pi_{\disk}\circ u:\Sigma \to \disk$ is holomorphic, where $\pi_{\disk}:\widehat{Y}_{\disk}\to Y_{\disk}\to \disk$ is the projection. 
\end{proposition}
\begin{proof}
Note that $\pi_{\disk}$ is $(J,\mathbf{i})$ holomorphic, where $\mathbf{i}$ is the standard complex structure on $\disk$. The assumption on $H$ implies that $X_H$ is parallel to $R=(1+\epsilon)^{-1}R_{\Gamma}$, which is projected to zero by $\pi_{\disk}$ (killed in the first projection of the composition). It is clear that $\pi_{\disk}\circ u$ is holomorphic. 

Note that 
$$u^*\rd \alpha(\partial_su,\partial_tu) = \rd\alpha (\partial_s u, J\partial_su)=(1+\epsilon)\rd\alpha_{\Gamma}(\pi_{\xi_{\Gamma}}\partial_s u,\pi_{\xi_{\Gamma}}J\partial_su)+\rd\left(\rho^2\rd \theta \right)(\pi_{\disk}\partial_su,\pi_{\disk} J\partial_su)$$
where $\pi_{\xi_{\Gamma}}$ is the natural projection to $\xi_{\Gamma}$ and $\pi_{\disk}$, by a bit abuse of notation, also denotes the natural projection to the $\disk$-direction on the tangent bundle. Since 
$$ \rd\alpha_{\Gamma}(\pi_{\xi_{\Gamma}}\partial_s u,\pi_{\xi_{\Gamma}}J\partial_su)=\rd\alpha_{\Gamma}(\pi_{\xi_{\Gamma}}\partial_s u,J\pi_{\xi_{\Gamma}}\partial_su)\ge 0$$
and 
\begin{eqnarray*}
    \rd\left(\rho^2 \rd \theta \right)(\pi_{\disk}\partial_su,\pi_{\disk} J\partial_su) & = & \rd\left(\rho^2\rd \theta \right)(\pi_{\disk}\partial_su,\mathbf{i}\pi_{\disk} \partial_su) \\
    & = & \rd\left(\rho^2\rd \theta \right)(\partial_s(\pi_{\disk}\circ u),\mathbf{i}\partial_s(\pi_{\disk}\circ u)) \\
    & = & \rd\left(\rho^2\rd \theta \right)(\partial_s(\pi_{\disk}\circ u),\partial_t(\pi_{\disk}\circ u))\\
    & = & (\pi_{\disk}\circ u)^* \rd\left(\rho^2\rd \theta \right)(\partial_su,\partial_tu)\ge 0.
\end{eqnarray*}
\end{proof}

\subsection{An $L$-simple holomorphic foliation}\label{ss:foliation}
We will describe a holomorphic foliation from \cite{avdek2024bourgeois}, which will help us confine the homology class of holomorphic planes in $\widehat{V\times \disk}$. Such an almost complex will have to be $\alpha$-tame $L$-simple instead of being compatible.

First, we fix a compatible almost complex structure $J_V$ on $\widehat{V}$. We then describe the almost complex structure $J^F$ on the positive end $\R_{+}\times Y_f$ of $\widehat{V\times \disk}$ by separating it into three regions $Y_V,Y_{\disk}$ and $Y_{\cap}$.
\subsubsection{$J^F$ on $Y_V$}
The contact hypersurface is given by $\rho^2=f$, where $f-1$ is $C^2$-small. We parametrize $Y_V$ by $V\times \Sphere^1$, then the Reeb vector field is
$$R=(f-\lambda(X_f))^{-1}(\partial_{\theta}-X_f).$$
The almost complex structure on $Y_V$ is described by 
$$J^F(R)=-C\partial_s,C>0, \quad J^F|_{TV}=J_V.$$
The purpose of positive constant $C$ is to make sure $J^F$ is tamed by $\rd(e^s\alpha)$, which holds for $C\ll 1$ by \cite[Lemma 3.2]{avdek2024bourgeois}. We require the almost complex structure $J_V$ near the critical point $p$ of $f$ to be standard in the local Darboux chart to set up the $L$-simple contact form in \S \ref{ss:L-simple-contact}. Then by \Cref{rmk:ana}, we have enough to set up the asymptotic analysis and Fredholm properties for curves asymptotic to $\gamma_p$.

The following proposition is straightforward from construction.
\begin{proposition}
$J^F$ is $\alpha$-tame for $\xi_{J^F}=TV$ and $4\pi$-simple.    
\end{proposition}
Let $\cO$ be a neighborhood of $\partial V$ in $V$ where $f$ only depends on $r$, we view $\cO \times S^1$ as subset of $\widehat{V}\times \C$ using the graph of $f$, and write $p=\rho^2$. Then we have
$$\xi_{J^F} = \xi_{\Gamma} \oplus \langle \partial_r+f'(r)\partial_p, R_{\Gamma} \rangle$$
as a subspace of $T(\widehat{V}\times \C)$.

\subsubsection{$J^F$ on $Y_{\disk}$}
Over $Y_{\disk}$, we pick a good almost complex structure $J^F$ using $J_V$. Hence $J^F$ is compatible, in particular, it is $\alpha$-tame with 
$$\xi_{J^F}=\xi=\xi_{\Gamma}\oplus \left\langle \partial_x+\frac{y}{1+\epsilon}R_{\Gamma}, \partial_y-\frac{x}{1+\epsilon}R_{\Gamma}\right\rangle.$$

\subsubsection{$J^F$ on $Y_\cap$}\label{sss:J^F}
Over $Y_\cap$, with the parametrization from $(1,1+\epsilon)\times \Gamma \times \SS^1$, the Reeb vector is 
$$R=(f(r)-rf'(r))^{-1}(\partial_{\theta}-f'(r)R_{\Gamma}).$$
As $f$ only depends on $r$, $R_{\Gamma}$ does not have $\partial_r$ component, the above formula can also be viewed as a vector in $T(\widehat{V}\times \C)$ when we view $Y_\cap$ as a subset in $\widehat{V}\times \C$. The contact structure $\xi$ is 
$$\xi = \xi_{\Gamma}\oplus \langle r\partial_{\theta}-f(r)R_{\Gamma}, \partial_r \rangle.$$
When we view $\xi$ as a subspace from the embedding, we have
$$\xi=\xi_{\Gamma}\oplus \left\langle R_{\Gamma}-\frac{r}{f(r)}\partial_{\theta}, \partial_r+f'(r)\partial_p \right\rangle.$$
Now we choose a function $h:(1,1+\epsilon) \to \R_{\ge 0}$ such that $h'\ge 0$, $h$ is $0$ near $1$ and is $1$ near $1+\epsilon$, then we define 
$$\xi_{J^F}:=\xi_{\Gamma} \oplus \left\langle \partial_r+f'(r)\partial_p, R_{\Gamma}-\frac{h(r)r}{f(r)}\partial_{\theta}\right\rangle$$
as a subspace. Then $\xi_{J^F}$ on $Y_\cap$ interpolates between $\xi_{J^F}$ on $Y_V$ and $\xi_{J^F}=\xi$ on $Y_{\disk}$.

We define $J^F$ by 
\begin{enumerate}
    \item $J^FR=-C\partial_s$, $C>0$;
    \item $J^F|_{\xi_{\Gamma}}=J_V$;
    \item $J^F(r\partial_r+rf'(r)\partial_p)=g(r)(R_{\Gamma}-\frac{h(r)r}{f(r)}\partial_{\theta})$, where $g(r)$ is a smooth function such that $g'\ge 0$, $g=1$ near $1$ and $g=-\frac{f'(r)}{2h(r)}=-\frac{f'(r)}{2}$ for $r$ near $1+\epsilon$.
\end{enumerate}
It is clear from the definition that $J^F$ patches with the almost complex structure on $Y_V$ smoothly. It also matches with $J^F$ on $Y_{\disk}$. This follows from the fact that $J^F$ over $Y_{\disk}$ is standard on the $\disk$-component by \eqref{eqn:J_good} and the standard complex structure $\mathbf{i}$ is $\mathbf{i}\partial_{\theta}=-2p\partial_p$ using the $(p,\theta)$ coordinate. It is straightforward to check that they match up.

\begin{proposition}
    $J^F$ is $\alpha$-tame on $Y_{\cap}$. 
\end{proposition}
\begin{proof}
    If we pull back everything using the parametrization  $(1,1+\epsilon)\times \Gamma \times \Sphere^1$, we have 
    $$\xi_{J^F}= \xi_{\Gamma}\oplus \left \langle r\partial_r, R_{\Gamma}-\frac{h(r)r}{f(r)}\partial_{\theta} \right \rangle$$
    with a contact form $\alpha = f\rd \theta+r\alpha_{\Gamma}$. For $X\in \xi_{\Gamma}$ and $Y\in \left \langle r\partial_r, R_{\Gamma}-\frac{h(r)r}{f(r)}\partial_{\theta} \right \rangle$, we have
        $$\rd \alpha (X+Y,J^FX+J^FY)=r\rd\alpha_{\Gamma}(X,J_VX)+\rd r \alpha_{\Gamma}(Y,J^FY)+f'(r)\rd r \rd \theta (Y,J^FY).$$
    It suffices to prove 
    $$(\rd r \alpha_{\Gamma}+f'(r)\rd r \rd \theta )(r\partial_r, J(r\partial_r))>0$$
    Indeed, it is $rg(r)-rf'(r)g(r)\frac{h(r)r}{f(r)}>0$ as $f'(r)<0$.
\end{proof}

\begin{proposition}\label{prop:direction}
    Over $Y_\cap\cup Y_{\disk}\subset \widehat{V}\times \C$, we have the $\partial_r$ component of $J^F R_{\Gamma}$ is negative.
\end{proposition}
\begin{proof}
    As $\partial_s=r\partial_r+p\partial_p$ viewed as in $\widehat{V}\times \C$, then over $Y_\cap$, we have $J^FR_{\Gamma}=-(1+\epsilon)C\partial_s$ having a negative component in the $\partial_r$ direction. Over $Y_{\disk}$, we can solve that 
    $$\left(-\frac{ghr}{f}f'+g\right)J^FR_{\Gamma}=-C\frac{ghr}{f}(f-rf')(r\partial_r+p\partial_p)-r\partial_r+rf'\partial_p$$
    The claim follows.
\end{proof}

\begin{proposition}\label{prop:foliation}
    There exists an extension of the above almost complex from the positive end $\R_{\ge 0}\times Y_f$ to $\widehat{V\times \C}$, such that it is tamed by the symplectic form and there is a holomorphic foliation by $\C$-family of leaves biholomorphic to $(\widehat{V},J_V)$. 
\end{proposition}
\begin{proof}
    We assume $f(1)=1+\eta$. We pick a smooth function $b:[0,\eta]\to \R$ such that $b'\ge 0$, $b=0$ near $0$ and $b=1$ near $\eta$. We first define $J^F$ on the domain in $\widehat{V}\times \C$ bounded by $p=f$ and $r=1$. On $V\times \disk$, we have $J^F=J_V\oplus \mathbf{i}$, where $\mathbf{i}$ is the standard almost complex structure on $\disk$. The remaining part is foliated by $p=b(\sigma)(f-1)+1$ for $\sigma\in (0,\eta)$. Using the parametrization $V\times \SS^1$ from the graph of $b(\sigma)(f-1)+1$, we define 
    $$J^F|_{TV}=J_V, \quad J^F R_{\sigma}=-Cb(\sigma)(r\partial_r+p\partial_p)+\frac{1-b(\sigma)}{1+\sigma}2p\partial_p$$
    where $R_{\sigma}$ is the Reeb vector on the hypersurface $p=b(\sigma)(f-1)+1$.
    It is straightforward to check that they patch up with contact structure on $\R_+\times Y_V\cup V\times  \disk$. For $\rd f$ sufficiently small and $C\ll 1$, we have $J^F$ tamed by the symplectic form. The missing domain $\mathbf{R}$ is bounded by $r=1,1+\epsilon$ and $p=f$. Over this region, we have $J^F|_{\xi_{\Gamma}}=J_V$, the rest direction is spanned by $\partial_r,R_{\Gamma},\partial_p, \partial_\theta$. We use $t$ to denote the local coordinate from integrating $R_{\Gamma}$, the symplectic form in this region is $r\rd \alpha_{\Gamma}+\rd r \rd t+\rd p \rd \theta$. Hence $\xi_{\Gamma}$ and $\langle  \partial_r,R_{\Gamma},\partial_p, \partial_\theta \rangle$ are symplectically orthogonal to each other. We only need to define $J^F$ on $\langle  \partial_r,R_{\Gamma},\partial_p, \partial_\theta \rangle$, tamed by the standard symplectic form $\rd r \rd t+\rd p \rd \theta$. We require that $J^F$ in those directions only depend on $p,r$ coordinates, in particular, it is invariant under the flow of $R_{\Gamma}$. This is possible as near the boundary of $\mathbf{R}$, such a property is enjoyed by the already defined $J^F$. The tameness $(\rd r\rd  t+ \rd p \rd \theta )(R_{\Gamma},J^F R_{\Gamma})>0$ implies that the $\partial_r$-component of $J^F R_{\Gamma}$ is negative.

    Now in the union of $\R_+\times Y_V$ and $r\le 1$, we have a foliation, with each leaf biholomorphic to $(V,J_V)$. On $V\times \disk$, the leaf is $V\times \{*\}$; On $\R_+\times Y_V$, the leaf is the surface with $s,\theta$ constant, i.e.\ the graph of $f$; On the middle region, the leaf is the graph of $b(\sigma)(f-1)+1$ with fixed $\theta$. We can extend and complete this foliation by adding the flow lines of $-J^FR_{\Gamma}$, which goes to $\infty$ in the $r$-coordinate by the discussion above and \Cref{prop:direction}. It is clear from the construction, $\widehat{V\times \disk}$ is covered by those leaves, each leaf is biholomorphic to $(\widehat{V},J_V)$ and the leaf space is diffeomorphic to $\C$, e.g.\ from the $\C$-fiber of $\widehat{V}\times \C$.
 \end{proof}

\begin{figure}[htb] {\small
\begin{overpic}[scale=0.5]
{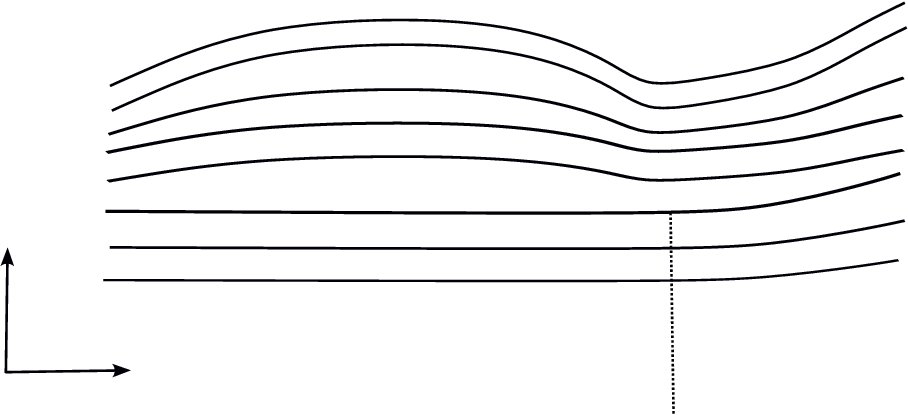}
\put (12,0) {$\partial_r$}
\put (-6,18) {$\partial_p$}
\put (75,10) {$r=1$}
\put (0,22) {$p=1$}
\put (45,45) {$\mathrm{graph}(f)$}
\put (0,30) {leaves}
\put (40,10) {$V\times \disk$}
\end{overpic}}
\caption{Foliation on $\widehat{V\times \disk}$}
\label{fig:fol}
\end{figure}

\begin{remark}
    Constructions of foliations in \cite{avdek2024bourgeois} were more involved, as we needed the leaves asymptotic to finite contact submanifolds in $Y_{\disk}$. While in \Cref{prop:foliation}, the leaves do not have nice asymptotics,  see \Cref{fig:fol}. But since we will only apply the intersection theory of cures asymptotic to $\gamma_p$ with this foliation to confine its homology class in \Cref{cor:disk}, the asymptotics of leaves do not matter.
\end{remark}

\begin{cor}\label{cor:disk}
Let $u:\C \to \widehat{V\times\disk}$ be a $J^F$ holomorphic plane asymptotic to $\gamma_p$, then $u$ intersects each leaves from \Cref{prop:foliation} transversely once. In particular, the homotopy class (relative to the boundary/asymptotics) of $u$ is given by the $\C$-fiber.
\end{cor}
\begin{proof}
     $\gamma_p$ placed at the positive end of $\widehat{V\times \disk}$ has linking number $1$ with any leaf from the foliation in \Cref{prop:foliation}. By the positivity of (interior) intersection, we have $u$ intersecting every leaf transversely at one point. In particular, $u$ is diffeomorphic to the leaf space, hence is homotopic to the $\C$-factor in $\widehat{V\times \disk}$. 
\end{proof}

With a bit of misuse of terminology, we use $J^F$ to denote the $\R$-invariant almost complex structure on $\widehat{Y}_f$ from $J^F$ restricted to the positive end of $\widehat{V\times \disk}$.
\begin{proposition}\label{prop:foliation_symp}
    $(\widehat{Y}_f,J_F)$ has a holomorphic foliation, with $\disk$-family of leaves biholomorphic to $\widehat{\Gamma}$, and a $\C\backslash \disk(\simeq\C^*)$-family of leaves biholomorphic to $\widehat{V}$.
\end{proposition}
\begin{proof}
    We can complete the foliations on $\widehat{Y}_V$ with $V$-leaves by adding flow lines of $-JR_{\Gamma}$ to get a $\C^*$ family of leaves biholomorphic to $\widehat{V}$. On the complement, by integrating flows of $-JR_{\Gamma}$ from $\Gamma \times \{(r,p,\theta)\}\in \widehat{V}\times \C$, we get a disk family of leaves biholomorphic to $\widehat{\Gamma}$.
\end{proof}

In the following, we describe an interpolation between $J^F$ on $\widehat{Y}_f$ and a compatible almost complex structure on $\widehat{Y}_{f}$, which then gives a tamed almost complex structure on the trivial cobordism. Over $Y_{V}$, using the parametrization $V\times \SS^1$, we have 
$$\xi=\left\langle X-\frac{\lambda(X)}{f}\partial_{\theta}\left| X\in TV \right.\right\rangle. $$
We can define $\xi^t$ as a family of hyperplane distributions on $Y_f$ such that
\begin{enumerate}
    \item $\xi^0=\xi_{J^F}$ and $\xi^1=\xi$;
    \item $\xi^t=\xi$ over $Y_{\disk}$;
    \item $\xi^t=\left\langle X-t\frac{\lambda(X)}{f}\partial_{\theta}| X\in TV \right \rangle$ over $Y_V$;
    \item $\xi^t=\xi_{\Gamma} \oplus \left\langle \partial_r+f'(r)\partial_p, R_{\Gamma}-\frac{(1-t)h(r)r+tr}{f(r)}\partial_{\theta}\right\rangle$ over $Y_\cap$.
\end{enumerate}
We also define similarly $J^t$ on $\xi^t$ by
\begin{enumerate}
    \item $J^t$ is the good $J$ on $Y_{\disk}$ and is independent of $t$;
    \item $J^t(X-t\frac{\lambda(X)}{f}\partial_{\theta})=J_VX-t\frac{\lambda(J_VX)}{f}\partial_{\theta}$ over $Y_V$.
    \item $J^t|_{\xi_{\Gamma}}=J_V$, $J^t(r\partial_r+rf'(r)\partial_p)=g(r)(R_{\Gamma}-\frac{(1-t)h(r)r+tr}{f(r)}\partial_{\theta})$ over $Y_\cap$.
\end{enumerate}

\begin{proposition}\label{prop:interpolation}
    $J^t$ is tamed by $\rd \alpha$ on $\xi^t$.
\end{proposition}
\begin{proof}
    It suffices to check the tameness on $Y_V$ and $Y_\cap$, both using their parameterizations. Over $Y_V$, 
    $$\rd\alpha\left(X-t\frac{\lambda(X)}{f}\partial_{\theta},J^t\left(X-t\frac{\lambda(X)}{f}\partial_{\theta}\right)\right)=\rd \lambda (X,J_VX)+\rd f \rd \theta \left(X-t\frac{\lambda(X)}{f}\partial_{\theta}, J_VX-t\frac{\lambda(J_VX)}{f}\partial_{\theta}\right).$$
    As long as $\rd f$ is sufficiently small, it is positive for $X\ne 0$ and $t\in [0,1]$.
    Over $Y_\cap$, $\alpha = f\rd \theta+r\alpha_{\Gamma}$. For $X\in \xi_{\Gamma}$ and $Y\in \left \langle r\partial_r, R_{\Gamma}-\frac{(1-t)h(r)r+tr}{f(r)}\partial_{\theta} \right \rangle$, we have
        $$\rd \alpha (X+Y,J^tX+J^tY)=r\rd\alpha_{\Gamma}(X,J_VX)+(\rd r \alpha_{\Gamma}+f'(r)\rd r \rd \theta) (Y,J^tY).$$
    It suffices to prove 
    $$(\rd r \alpha_{\Gamma}+f'(r)\rd r \rd \theta) (r\partial_r, J^t(r\partial_r))>0$$
    Indeed, it is $rg(r)-rf'(r)g(r)\frac{(1-t)h(r)r+tr}{f(r)}>0$. 
\end{proof}

Using such an interpolation $(\xi^t,J^t)$, we can find an almost complex structure $J^I$ on $\widehat{Y}_f$, such that 
$J^IR=-C(s)\partial_s$, $J|_{\xi^t}=J^t$ on the slice of $\{t\}\times Y_f\subset \widehat{Y}_f$ and $J^I=J^1,J^0$ on slices above $s=1$ and below $s=0$ respectively. For $C(s)\ll 1$ when $s\le 1$, we get an almost complex structure on the symplectization tamed by $\rd (e^s\alpha)$ with cylindrical almost complex structure at two ends and $\rd\alpha(X,JX)\ge 0$. As a consequence, we have $\int u^*\rd \alpha\ge 0$ for any holomorphic curve $u:\Sigma \to \widehat{Y}_f$, with inequality holds if and only if $u$ is contained in $\R_s\times \mathrm{im}(\gamma)$ for a Reeb orbit $\gamma$ as the usual case.

\section{Holomorphic curves}\label{s4}
\subsection{Holomorphic curves and intersection number}
Let $V$ be a Liouville domain and $U$ a codimension $2$ symplectic submanifold as in \Cref{thm:V}. We write $Y:=\partial(V\times \disk)$ and $H:=\partial(U\times \disk)$ with contact forms described in \S \ref{s3}. For a critical point $p$ of $f$, which gives a simple Reeb orbit $\gamma_p\subset V\times \Sphere^1\subset \partial(V\times \disk)$, we fix a parametrization of $\gamma_p$ and define the moduli space
$$\cM_J(\gamma_p):= \left\{u:\C \to \widehat{V\times \disk},\theta \in \Sphere^1 \left|\begin{array}{c}
     \overline{\partial}_J u=0,[\pi\circ u]=0\in H_2(V;\Q),  \\
     \displaystyle \lim_{s\to \infty}u(s+\mathbf{i}t)=\gamma_p(Tt+\theta) 
\end{array} \right.\right\}/\mathrm{Aut}(\C),$$
where we use $e^{2\pi(s+\mathbf{i}t)}$ coordinate on $\C$, $T$ is the period of $\gamma_p$ and $\pi$ is the projection from $\widehat{V\times \disk}=\widehat{V}\times \C$ to $\C$.
$\pi\circ u$ induces a map $\Sphere^2\to \widehat{V}$ and $[\pi\circ u]$ is the pushforward of the fundamental class of $[\Sphere^2]$. The assumption $[\pi\circ u]=0$ is equivalent to that $u$ is homologous (relative to the boundary/asymptotics) to the  $\C$-fiber $\widehat{V}\times \C$. For $q\in \widehat{V\times \disk}$, we define
$$\cM_J(\gamma_p,q):= \left\{u:\C \to \widehat{V\times \disk },\theta\in \Sphere^1 \left|
\begin{array}{c}
  \overline{\partial}_J u=0, u(0)=q, [\pi\circ u]=0\in H_2(V;\Q),\\
   \displaystyle \lim_{s\to \infty}u(s+\mathbf{i}t)=\gamma_p(Tt+\theta)
\end{array}
\right.\right\}/\mathrm{Aut}(\C,0),$$
where $\mathrm{Aut}(\C,0)$ is the automorphism group of $\C$ preserving $0$.  A priori, the virtual dimension of $\cM_J(\gamma_p)$ depends on the homology class of $u$ if $c_1(V)\ne 0$. If $[\pi\circ u]=0$, then $\virdim \cM_J(\gamma_p,q)= \ind(p)-\dim V$, where $\ind(p)$ is the Morse index of $p$, see \cite[Theorem 6.3]{ADCI}.  

More generally, let $\rho:M\to \widehat{V\times \disk}$ be a smooth map from a manifold $M$, we define 
$$\cM_J(\gamma_p,\rho):= \left\{\begin{array}{c}
     u:\C \to \widehat{V\times \disk },  \\
     m\in M,\theta \in \Sphere^1
\end{array}\left|
\begin{array}{c}
  \overline{\partial}_J u=0,u(0)=\rho(m),[\pi\circ u]=0\in H_2(V;\Q), \\
  \displaystyle \lim_{s\to \infty}u(s+\mathbf{i}t)=\gamma_p(Tt+\theta)
\end{array}
\right.\right\}/\mathrm{Aut}(\C,0).$$
As $\gamma_p$ is a simple Reeb orbit, the above moduli spaces can be assumed to be cut out transversely for a generic choice of $J$ by the somewhere injective property. We use $\overline{\cM}_{\bullet}(\bullet,\bullet)$ to denote a compactification, which is the canonical SFT compactification of $\cM_{\bullet}(\bullet,\bullet)$ if $M$ is compact. In general, a compactification depends on a compactification of $M$ in the context.
\begin{proposition}\label{prop:intersection}
Let $W$ be a symplectic filling of $H$ inside $V\times \disk$. For any $L$-simple almost complex structure $J$ on $\widehat{V\times \disk}$ making $\widehat{W}$ into a holomorphic hypersurface as in \S \ref{s3}, we have $u\bullet W=0$ for $u\in \cM_J(\gamma_p)$. Then $u$ is either contained in $\widehat{W}$ or $u$ does not intersect $\widehat{W}$ by \Cref{cor:leaf}.
\end{proposition}
\begin{proof}
If $p\notin U$, by pushing $\gamma_p$ to a $\Sphere^1$ fiber over a point in $\partial V$ following a path from $p$ to the point in $\partial V$ in $V$ without touching $U$, we can find a bounding disk $v$, which is (rationally) homologous to $u$ and is contained in $\partial(V\times \disk)$ and does not intersect $W$, As a consequence, $0=u\cdot \widehat{W}= u\bullet W$. In this case, $u$ does not intersect $\widehat{W}$. When $p\in U$, by \Cref{ex:eigen}, $u_{\tau}$ is the push of $u$ in the normal direction of $U\subset V$ at $p$. Hence by a similar argument as in $p\notin U$ case, we have $u_{\tau} \cdot \widehat{W}=0$ and hence $u\bullet W=0$.
\end{proof}

\begin{remark}
    We neither assume $c_1(V)=0$ nor $H_2(V;\Q)=0$ in \Cref{thm:V}, even though they are both satisfied in the standard sphere case. Because of this, we need to keep track of the homology class of the holomorphic curves as above for the purpose of virtual dimensions as well as intersection numbers above. This leads to the complication of introducing the foliation in \S \ref{ss:foliation}.
\end{remark}

\subsection{Linearized contact homology and $S^1$-equivariant symplectic cohomology}
\subsubsection{Basics of the two invariants} Here we briefly recall the concept of contact homology \cite{BH,Pardon} and linearized contact homology. Let $(Y,\xi)$ be a $2n-1$ dimensional contact manifold. We fix a contact form $\alpha$ such that all Reeb orbits are non-degenerate. For each good Reeb orbit $\gamma$ we associate it with a formal variable $q_{\gamma}$. We use $V$ to denote the $\Z/2$ graded $\Q$ vector space generated by $q_{\gamma}$, with grading $|q_{\gamma}|=\mu_{CZ}(\gamma)+n-3$. By counting rigid punctured holomorphic curves with one positive puncture and multiple negative punctures in the symplectization $\widehat{Y}=\R\times Y$, we can get degree $-1$ operations
$$\partial_k:V\to S^kV=(\otimes^k V)/\Sigma_k$$
such that $\partial = \sum_{k=0}^\infty \partial_k$ on $SV=\oplus_{k=0}^\infty S^kV$ defined by the Leibniz rule 
$$\partial(ab)=\partial(a)b+(-1)^{|a|} a\partial(b)$$
determines a differential graded algebra (dga) $(SV,\partial)$. The homology of $(SV,\partial)$ is called the contact homology of $Y$, which is a contact invariant \cite{BH,Pardon}. 

An augmentation $\epsilon$ is a dga map from $(SV,\partial)$ to $(\Q,0)$, i.e.\ a map $\epsilon:V\to \Q$ such that $\sum \epsilon^k \circ \partial_k:V\to \Q$ is zero. Given an augmentation $\epsilon$, we have an algebra isomorphism $F_{-\epsilon}:SV\to SV$ determined by $F_{-\epsilon}(x)=x-\epsilon(x)$. The inverse of $F_{-\epsilon}$ is $F_{\epsilon}$. Under such change of coordinate, $F_{\epsilon}\circ \partial \circ F_{-\epsilon}$ is again induced from a sequences of operations $\partial_{\epsilon,k}:V\to S^kV$ with the significance that $\partial_{\epsilon,0}=0$. As a consequence, we have $\partial_{\epsilon,1}$ is a differential on $V$. Then we can define the linearized contact homology w.r.t.\ $\epsilon$ by $\LCH_{*}(Y,\epsilon)=H_*(V,\partial_{\epsilon,1})$. The linearized contact homology certainly depends on the augmentation, but the set of linearized contact homology after enumerating through all augmentations is a contact invariant by \cite{chaidez2024contact}.

A closely related invariant called $\SS^1$-equivariant positive symplectic cohomology was originally introduced by Viterbo \cite{zbMATH01398851}, described in an alternative way by Seidel \cite{zbMATH05504309}, and systematically defined and studied by Bourgeois and Oancea \cite{BO,BO_Gysin,BO_S1}. For this, we need a Liouville filling $W$ of the contact manifold. By considering a quadratic Hamiltonian on the competition $\widehat{W}$, which is $C^2$-small on $W$, the $\SS^1$-equivariant Hamiltonian-Floer cohomology generated by the non-constant Hamiltonian orbits is the $\SS^1$-equivariant positive symplectic cohomology $SH^*_{+,\SS^1}(W)$. It is $\Z/2$ graded by $n-\mu_{CZ}$. $\SS^1$-equivariant positive symplectic cohomology of Liouville fillings can be defined over $\Z$ and hence any coefficient ring. We refer readers to \cite{BO_Gysin,BO_S1,CO,zhao} for its constructions, properties, as well as various variants of the invariant.
Both linearized contact homology and $\SS^1$-equivariant symplectic cohomology have a filtered version, i.e.\ the (co)homology of the subcomplex generated by Reeb orbits with period up to a threshold $D$. We use $\LCH_{*}^{<D}(Y,\epsilon)$ to denote the filtered linearized contact homology and $SH^{*,<D}_{+,\SS^1}(W)$ the filtered $\SS^1$-equivariant positive symplectic cohomology. As both invariants use the information of period, they depend on the contact form on $Y$ and $\partial W$ respectively. 

From the tautological long exact sequence, we have a connecting map $SH_{+,\SS^1}^*(W)\to SH_{0,\SS^1}^*(W)[1]=H^*(W)[1]\otimes \Z[u,u^{-1}]/\langle u \rangle =H^*(W)[1]\otimes \Z[u^{-1}]$, where $\deg u=2$. We can project it further to $H^*(W)$ to obtain a map $SH_{+,\SS^1}^*(W)\stackrel{\tau}{\to}H^{*}(W)[1]$. A similar map can be defined for linearized contact homology as well. Let $W$ be a Liouville filling. We can put a Morse function $g$ on $W$ such that the gradient of $g$ intersects transversely with $\partial W$ and points outward. By counting holomorphic planes in $\widehat{W}$ with one marked point intersect the stable manifold of a critical point (i.e.\ those points converging to the critical point following the gradient flow), we can define a map $\LCH_*(Y,\epsilon_W) \stackrel{\tau}{\to} H^{2n-2-*}(W)$. We will review all of these moduli spaces in the special setting of $W=V\times \disk$ in the proof of \Cref{thm:BO} below.

\subsubsection{The Bourgeois-Oancea isomorphism}
Following the functoriality of contact homology, one geometric way to obtain an augmentation is via counting holomorphic planes in a Liouville filling. In \cite{BO}, Bourgeois and Oancea showed that $\LCH_{2n-3-*}(Y,\epsilon_{W})$ is isomorphic to $SH^{*}_{+,\SS^1}(W)$, where $\epsilon_W$ is the augmentation from the filling $W$, see also \cite{BO_S1}.

The following is a filtered version of  Bourgeois and Oancea's theorem applied to the Liouville domain $V\times \disk$ with contact form described in \S \ref{s3} along with maps to $H^*(V\times \disk)$, which does not appeal to any virtual techniques because of the period threshold.
\begin{theorem}[\cite{BO}]\label{thm:BO} 
We have the following commutative diagram
    $$
    \xymatrix{
     \LCH^{<2\pi+\delta}_{2n-3-*}(Y_f,\epsilon_{V\times \disk}) \ar[r]^{\simeq }_{\Phi}\ar[d]^{\tau} & SH^{*,<2\pi+\delta}_{+,\SS^1}(V\times \disk) \ar[d]^{\tau}\\
     H^{*+1}(V\times \disk) \ar[r]^{=} & H^{*+1}(V\times \disk)
    }
    $$
    where all cohomology/homology are defined over $\Z$. 
\end{theorem}
\begin{proof}
The proof is essentially contained in \cite{BO,BO_S1}. In the following, we will review the arguments in \cite{BO,BO_S1} and explain how it works for \Cref{thm:BO}. Some notational discrepancies and further comments are contained in \Cref{rmk:furthur} below.
\begin{enumerate}
    \item\label{step1} With a contact form in \S \ref{s3}, $\LCH^{<2\pi+\delta}_{2n-3-*}(Y_f,\epsilon_{V\times \disk})$ involves only simple Reeb orbits, therefore all the moduli spaces involved can be arranged to be cut out transversely for a generic choice of $J$. Moreover, as none of the curves can be multiply covered, the curve counting can be defined over $\Z$.  More precisely, the chain complex is the $\Z$-module generated by $q_{\gamma_{p}}$. Throughout the proof, we fix a parametrization of the Reeb orbit $\gamma_p$. The differential from $q_{\gamma_{p_+}}$ to $q_{\gamma_{p_-}}$ is the counting of rigid points in
    \begin{equation}\label{eqn:LCH}
      \left\{u:\R_s\times \SS^1_t \to \widehat{\partial(V\times \disk)},\theta_{\pm}\in \SS^1 \left|\begin{array}{c}
         \partial_su+J\partial_t u=0,  \\
         \displaystyle \lim_{s\to \pm \infty } u(s,t) =(\pm \infty, \gamma_{p_{\pm}}(T_{\pm}t+\theta_{\pm}))
    \end{array}  \right.\right\}/(\C^*\times \R)  
    \end{equation}
    Here $T_{\pm}$ are the periods of $\gamma_{p_{\pm}}$, $\theta_{\pm}$ are the two asymptotic markers at two ends, the $\C^*$-action is the automorphism group of $\R\times S^1$ which also acts on the asymptotic markers, the $\R$-action is the translation action on the target $\widehat{\partial(V\times \disk)}$. As all Reeb orbits have periods approximately $1$, there is no appearance from the augmentation $\epsilon_{V\times \disk}$.
    \item One way to define $SH^{*,<2\pi+\delta}_{+,\SS^1}(V\times \disk)$ is as follows. Consider a Hamiltonian $H$ on the completion $\widehat{V\times \disk}=(V\times \disk)\cup (1,+\infty)_r\times Y_f$ that is zero $V\times \disk$ and $H=h(r)$ on $(1,+\infty)_r\times Y_f$, such that $h'(r)>0,h''(r)\ge 0$ on $(1,\infty)$ and $h'(r)=1$ for $r\ge 1+\delta$. The non-constant Hamiltonian orbits of $X_H$ come in $\SS^1$-families which are one-to-one corresponding to Reeb orbits of $Y_f$ of period at most $1$, i.e.\ those $\gamma_p$ in \S \ref{s3}. We use $\overline{\gamma}_p$ to denote the corresponding $\SS^1$-family of Hamiltonian orbits. The cochain complex for $SH^{*,<2\pi+\delta}_{+,S^1}(V\times \disk)$ are the $\Z$-module generated by $\overline{\gamma}_p$, with differential from $\overline{\gamma}_{p_+}$ to $\overline{\gamma}_{p_-}$ is defined by counting rigid points in
    \begin{equation}\label{eqn:S^1}
        \left\{u:\R_s\times \SS^1_t \to \widehat{V\times \disk}\left|\begin{array}{c}
         \partial_su+J(\partial_t-X_H(u))=0,  \\
         \displaystyle \lim_{s\to \pm \infty} u \in \overline{\gamma}_{p_{\pm}}
    \end{array} \right.\right\}/\R\times \SS^1
    \end{equation}
    Here we use a $t$-independent almost complex structure $J$, hence $\SS^1$ acts on the space of solutions by reparameterization in the $\SS^1$-direction. We can achieve transversality for those curves as all orbits are simple by \cite[Proposition 3.5 (i)]{BODuke}. The counting is again over $\Z$ as there is no orbifold point in the moduli space. This indeed computes $SH^{*,<2\pi+\delta}_{+,\SS^1}(V\times \disk)$, which is in generally true if we can achieve $\SS^1$-equivariant transversality (i.e.\ choosing a $t$-independent almost complex to achieve transversality), see \cite[\S 5.2.2, Proposition 5.9]{Zhou22} for the derivation of this fact using a cascades version of the Borel construction of the $\SS^1$-equivariant symplectic cohomology \cite{BO_S1}. 
    \item As done in \cite[\S 5]{BO}, we can apply neck-stretching to \eqref{eqn:S^1} along the contact boundary $Y_f$ to get an alternative definition of $SH^{*,<2\pi+\delta}_{+,\SS^1}(V\times \disk)$ by counting 
    \begin{equation}\label{eqn:S_1}
        \left\{u:\R_s\times \SS^1_t \to \widehat{\partial(V\times \disk)}\left|\begin{array}{c}
         \partial_su+J(\partial_t-X_H(u))=0,  \\
         \displaystyle \lim_{s\to \pm \infty} u \in \overline{\gamma}_{p_{\pm}}
    \end{array} \right.\right\}/\R\times \SS^1
    \end{equation}
    Here $H$ on $\widehat{Y}_f=\R_+\times Y_f$ is $h(r)$ such that $h=0$ for $r\in (0,1)$ and $h'(r)>0,h''(r)\ge 0$ on $(1,\infty)$ and $h'(r)=1$ for $r\ge 1+\delta$. In general, neck-stretching applied to \eqref{eqn:S^1} yields \eqref{eqn:S_1} possibly with negative punctures asymptotic to Reeb orbits. And the differential counts the punctured version of \eqref{eqn:S_1} capped off by counting rigid holomorphic planes in $\widehat{V\times\disk}$, that is the augmentation $\epsilon_{V\times \disk}$. In the special case here, there is no action room to form extra negative punctures; therefore, the differential ends up counting \eqref{eqn:S_1}. Similarly, no extra negative punctures are needed in any moduli spaces of curves in $\widehat{Y}_f$ explained below, which need to be considered and capped off by the augmentation in the general setting.  
   \item So far, the definition of $\LCH^{<2\pi+\delta}_{2n-3-*}(Y_f,\epsilon_{V\times \disk})$ is already very similar to the definition of $SH^{*,<1}_{+,\SS^1}(V\times \disk)$, except for the appearance of Hamiltonian for $SH^{*,<1}_{+,\SS^1}(V\times \disk)$. Following \cite[\S 6]{BO} (see \Cref{rmk:general} below for more detailed comparison), the map $\Phi$ from $q_{\gamma_{p_+}}$ to $\overline{\gamma}_{p_-}$is defined as the counting of 
    \begin{equation}\label{eqn:phi}
     \left\{u:\R_s\times \SS^1_t \to \widehat{\partial(V\times \disk)}, \theta_+\in \SS^1\left|\begin{array}{c}
         \partial_su+J(\partial_t-X_{H_s}(u))=0, \\
         \displaystyle \lim_{s\to \infty} u(s,t)=(\infty, \gamma_{p_+}(T_+t+\theta_+))\\
         \displaystyle \lim_{s\to - \infty} u \in \overline{\gamma}_{p_{-}}
    \end{array} \right.\right\}/\SS^1
    \end{equation}
    Here $H_s=\rho(s)H$ where $\rho=0$ for $s\gg 0$ and $\rho=1$ for $s \ll 0$ and $\rho'(s)\le 0$.  
    This map preserves the filtration induced from the periods of $\gamma_p$ and the leading term is $q_{\gamma_p}\mapsto \overline{\gamma}_p$ coming from a suitable reparameterization of the trivial cylinder over the orbit $\gamma_p$, see \cite[P.661-662]{BO}. This establishes the isomorphism $\Phi:\LCH^{<2\pi+\delta}_{2n-3-*}(Y_f,\epsilon_{V\times \disk}) \to SH^{*,<2\pi+\delta}_{+,\SS^1}(V\times \disk)$.
    \item The homotopy inverse to $\Phi$ can be defined similarly \cite[Remark 17]{BO}. More precisely, $\Psi$ from $\overline{\gamma}_{p_+}$ to $q_{\gamma_{p_-}}$ is defined by counting 
    \begin{equation}\label{eqn:phi-1}
     \left\{u:\R_s\times \SS^1_t \to \widehat{\partial(V\times \disk)}, \theta_-\in \SS^1\left|\begin{array}{c}
         \partial_su+J(\partial_t-X_{H}(u))=0, \\
         \displaystyle \lim_{s\to -\infty} u(s,t)=(-\infty, \gamma_{p_-}(T_-t+\theta_-))\\
         \displaystyle \lim_{s\to \infty} u \in \overline{\gamma}_{p_{+}}
    \end{array} \right.\right\}/\R\times \SS^1
    \end{equation}
    This map also preserves the period filtration and the leading term is $\overline{\gamma}_p\mapsto q_{\gamma_p}$, see \cite[Proposition 4.7]{zhou24}. It is indeed the inverse to $\Phi$ up to (filtered) homotopy (claimed in \cite[Remark 17]{BO}): the curve counting of $\Psi\circ\Phi$ corresponds to a Floer breaking of the following moduli space
    \begin{equation}\label{eqn:homotopy}
     \left\{u:\R_s\times \SS^1_t \to \widehat{\partial(V\times \disk)}, \theta_{\pm}\in \SS^1\left|\begin{array}{c}
         \partial_su+J(\partial_t-X_{H_s}(u))=0, \\
         \displaystyle \lim_{s\to \pm \infty} u(s,t)=(\infty, \gamma_{p_{\pm}}(T_{\pm}t+\theta_{\pm}))
    \end{array} \right.\right\}/\SS^1
    \end{equation}
    Besides the SFT breakings at the two ends corresponding to compositions with differentials in step \eqref{step1}, the only remaining end of \eqref{eqn:homotopy} is set of $u$ where the $r$ coordinate of $u(s_0)$ is smaller than $1$ where $\rho(s_0)=0$. By maximal principle, $r\circ u(s\le s_0)\le 1$, as a consequence, the equation is just $\partial_su+J\partial_tu=0$ in \eqref{eqn:LCH}. Therefore, the limit of such an end is precisely \eqref{eqn:LCH} without quotienting out the target $\R$-translation, i.e.\ rigid curves in the symplectization viewed as the completion of the trivial cobordism. In particular, counting those gives the identity map $q_{\gamma_p}\mapsto q_{\gamma_p}$ from the trivial cylinder, as other curves are never rigid from the $\R$-translation on the target. This shows that  $\Psi\circ\Phi$ is homotopic to $\Id$. As $\Phi,\Psi$ are isomorphisms from the filtration perspective, we know that $\Phi\circ\Psi$ is also homotopic to $\Id$.
    \item Finally, the map $\tau:SH^{*,<2\pi+\delta}_{+,\SS^1}(V\times \disk)\to H^*(V\times \disk)[1]$ can be defined as follows. We fix a Morse function $g$ on $V\times \disk$ such that the gradient of $g$ intersects $\partial(V\times \disk)$ transversely and points outward. For a critical point $q$, we use $S_q$ to denote the stable manifold of $q$, which is contained in the interior of $V\times \disk$. Then $\delta$ from $\overline{\gamma}_p$ to $q$ is defined by counting
     \begin{equation}\label{eqn:delta}
     \left\{u:\C \to \widehat{\partial(V\times \disk)} \left|\begin{array}{c}
         \partial_su+J(\partial_t-X_{H}(u))=0, \\
         u(0)\in S_q,
         \displaystyle \lim_{s\to \infty} u \in \overline{\gamma}_{p_{+}}
    \end{array} \right.\right\}/\R\times \SS^1
    \end{equation}
    Here the end of $\C$ is parametrized by $e^{2\pi (s+\mathbf{i}t)}$. We can apply neck-stretching to \eqref{eqn:delta}, which in the limit breaks into the product of \eqref{eqn:phi-1} with the following moduli space
    \begin{equation}\label{eqn:delta_LCH}
     \left\{u:\C \to \widehat{\partial(V\times \disk)}, \theta_-\in \SS^1\left|\begin{array}{c}
         \partial_su+J\partial_t=0,   u(0)\in S_q,\\
         \displaystyle \lim_{s\to \infty} u(s,t)=\gamma_{p_{-}(T_-t+\theta_-)}
    \end{array} \right.\right\}/\R\times \SS^1
    \end{equation}
    That is the moduli space defining $\tau:\LCH^{<2\pi+\delta}_{2n-3-*}(Y,\epsilon_{V\times \disk})\to H^{*+1}(V\times\disk)$. As a consequence, we have $\tau \circ \Phi^{-1}=\tau$ on homology. This proves the theorem.
\end{enumerate}
\end{proof}

\begin{remark}\label{rmk:furthur}
    A few remarks regarding notations in \cite{BO} and the proof of \Cref{thm:BO} are as follows:
    \begin{enumerate}
        \item The proof of \Cref{thm:BO} is greatly simplified as all Reeb orbits are simple and have approximately the same period. As a consequence, no augmentation curves are involved, and all moduli spaces are cut out transversely for a generic almost complex structure and are manifolds. In general, transversality assumptions \cite[Remark 9]{BO} were needed in \cite{BO}, or one needs to deploy suitable virtual machineries to execute the proofs in \cite{BO}.
        \item \cite{BO} used the convention $\omega(X_H,\cdot)=\rd H$ and ``non-standard"\footnote{Compared to conventions for symplectic cohomology/homology and contact homology used more often in recent literature, e.g.\ \cite{BH,CO,Pardon}} cylindrical coordinates near ends and asymptotic conditions \cite[P.628]{BO} to be consistent with the convention of the Hamiltonian vector field. One can translate between two conventions using the biholomorphism $\R\times \SS^1 \to \R\times \SS^1, (s,t)\mapsto (-s,-t)$.   
        \item \cite[\S 5-6, Proposition 4]{BO} were applied to a non-equivariant version of the linearized contact homology, whose chain complex is denoted by $BC_*(\lambda)$ in \cite{BO}. This is supposed to be the positive symplectic cohomology by \cite[Proposition 4]{BO}. The construction of positive symplectic cohomology in \cite[\S 2.2]{BO} was the cascades construction, called $BC_*(H)$,  built on \eqref{eqn:S^1} as $\SS^1$-equivariant transversality was assumed in \cite[Remark 9]{BO}. For the purpose of \Cref{thm:BO}, as we do not need the Gysin sequence, we can bypass the concept of positive symplectic cohomology, and compare $\LCH$ with $SH^*_{+,\SS^1}$ directly. 
        \item When we use an almost complex structure in \S \ref{s2}, since $\xi_J$ is $C^\infty$ close to $\xi$, the $r$-coordinate on the end of $\widehat{V\times \disk}$ is still strictly plurisubharmonic w.r.t.\ $J$, i.e.\ maximum principle for $r$ holds as well. That is, we can use such almost complex structures to define symplectic cohomology.
        \item One can use the $L$-simple almost complex structure $J^F$ from \S \ref{ss:foliation}, after a small perturbation to achieve transversality, to define $\LCH^{<2\pi+\delta}_*(Y_f,\epsilon_{V\times \disk})$ and $\tau$.
    \end{enumerate}
\end{remark}

\subsection{Holomorphic curves that probe the topology}
In the following, we will describe the holomorphic curve that will detect the topology of the knot complement. Since we will be using intersection theory, it is important to keep track of the homology classes of the holomorphic curves.

In \cite[\S 2]{Z23}, we computed $SH^{*,<2\pi+\delta}_{+}(V\times \disk) \to  H^{*}(V\times \disk)[1]$. The $\SS^1$-equivariant analogue can be done similarly. We use $p_{\max}$ to denote the unique local maximum of $f$, which is contained in $U\subset V$. The corresponding Reeb orbit $\gamma_{p_{\max}}$ has the maximal period among those $\gamma_p$ orbits. We first prove some preliminary results.

\begin{lemma}\label{lemma:non-trivial}
    Let $V$ be a compact manifold with a Riemannian metric $g$. Then there exists a constant $C>0$, such that if $u:M\to V$ from a closed manifold with $u_*[M] \ne 0 \in H_*(V;\Q)$, we have $\mathrm{vol}_g(u)>C$. The same holds for $M$ with boundary and $u:(M,\partial M)\mapsto (V,\partial V)$ representing a non-trivial class in $H_*(V,\partial V;\Q)$.
\end{lemma}
\begin{proof}
    We pick a basis $\{\alpha_i\}$ of $H^*(V;\R)$ using differential forms. Then there exists $D>0$ such that for any $p$ and $k$ linearly independent vectors $v_1,\ldots,v_k\in T_pV$, we have 
    $$|\alpha_i(v_1,\ldots,v_k)|<D\mathrm{vol}_g(v_1,\ldots,v_k).$$
    As a consequence, we have $D\mathrm{vol}_g(u)\ge \int |u^*\alpha_i|\ge |\int u^*\alpha_i|$ for all $i$. Since $u_*[M]$ is nontrivial rationally, there is a universal constant $C$ such that $\mathrm{vol}_g(u)>C$. We consider compactly supported differential forms for the case with boundary, the proof is identical.
\end{proof}

In the following, we will consider the following $\alpha$-complex structures $J$, such that 
\begin{enumerate}
   \item $J$ is good.
    \item $\xi_J$ is $\xi^t$ described before \Cref{prop:interpolation} for some $t$. And $J$ is induced from a $\lambda$-compatible almost complex structure $J_V$ on the Liouville domain $(V,\lambda)$ via the natural identification between $\xi_J$ and $TV$ by projection.
    \item Over $Y_{\cap}$, we have $\xi_J=\xi_{\Gamma}\oplus \langle \partial_r+f'(r)\partial_p, R_{\Gamma}-v(r)\partial_{\theta} \rangle$, where $v(r)\ge 0$. We require $J$ respects the splitting. And when we restrict $J$ to $\xi_{\Gamma}$, we get a $\rd\alpha_{\Gamma}$ compatible almost complex structure $J_{\Gamma}$ from the restriction of $J_V$. $J(r\partial_r+rf'(r)\partial_p)=g(r)(R_{\Gamma}-v(r)\partial_{\theta})$, where $g$ is defined in \S \ref{sss:J^F} (so that the almost complex structure is compatible with the good condition).
\end{enumerate}
Such collection of almost complex structures includes the almost complex structure $J^F$ with a holomorphic foliation, some special $\alpha$-compatible almost complex structures, as well as the interpolation between them considered in \Cref{prop:interpolation}. The purpose of working with such almost complex structures is just to obtain uniform estimates between $\int u^*\alpha $ and the area of a projection to $\widehat{V}$ in \Cref{prop:trivial_topology}, which certainly works for many other almost complex structures.

\begin{proposition}\label{prop:trivial_topology}
    Suppose $f-1$ in \S \ref{s3} is sufficiently $C^2$ small on $V$ and $J$ is an almost complex structure as above. For any solution $u\in \eqref{eqn:LCH},\eqref{eqn:S_1},\eqref{eqn:phi-1}$, we have $u$ does not enter $\widehat{Y}_{\disk}$ and $[\pi_V\circ u]=0\in H_2(V;\Q)$, where $\pi_V$ is the projection $\widehat{Y_{V}\cup Y_\cap }\to Y_{V}\cup Y_\cap \to \widehat{V}$. Here we assume the almost complex structure is good on $Y_{\disk}$ and can either be a compatible one or the $L$-simple almost complex structures used in \S \ref{ss:foliation} to build the foliation. 
\end{proposition}
\begin{proof} 
We first consider the case of \eqref{eqn:LCH}. As $\gamma_p$ has winding number $1$ with $\Gamma \times \{0\}\in Y_{\disk}$, by \Cref{prop:d}, the holomorphic map $\pi_{\disk}\circ u$ must have empty image as it misses $0$. This shows that $u$ does not enter $\widehat{Y}_\disk$. Now $\int u^*\rd \alpha$ is sufficiently small as $f-1$ in \S \ref{s3} is sufficiently $C^2$ small on $V$. As the complex structure respects the splitting $\xi_J\oplus \langle R,\partial_s \rangle$ and $\rd \alpha$ vanishes on $\langle R, \partial_s \rangle$, we have $u^*\rd \alpha$ only depends on the $\xi_J$ factor of $\rd u$.

On $\widehat{Y}_{V}$, the linearization of $\pi_V:Y_V\to V$ induces an isomorphism $\rd \pi_V|_{\xi_J}:\xi_J \to TV$. The metric $g_{\xi_J}=\rd \alpha(\cdot, J\cdot)$ on $\xi_J$ will then be pushed forward by $\rd \pi_V|_{\xi_J}$ to a metic defined by  
$$|X|^2=\rd \lambda (X,J_VX)-t(Xf)\cdot \frac{\lambda (J_VX)}{f}+t((J_VX)f)\cdot \frac{\lambda(X)}{f}, \quad X\in TV.$$
When $f-1$ is $C^2$ small, such metric is $C^1$ close to $g_V=\rd \lambda (\cdot, J_V \cdot)$ on $V$. As a consequence, for $f-1$ sufficiently $C^2$ small, we have a constant $C$ such that 
$$\mathrm{vol}_{g_V}(\pi_V\circ u|_{u^{-1}(\widehat{Y}_V)})\le C \int_{u^{-1}(\widehat{Y}_V)} u^*\rd \alpha.$$

On $\widehat{Y}_{\cap}$, we have $\alpha = r\alpha_{\Gamma}+f(r)\rd \theta$ and hence $\rd \alpha = r\rd \alpha_{\Gamma}+\rd r \wedge \alpha_{\Gamma}+f'(r)\rd r \wedge \rd \theta$. We have $\xi_J=\xi_{\Gamma}\oplus \langle \partial_r+f'(r)\partial_p, R_{\gamma}-v(r)\partial_{\theta}\rangle$.The linearization of the projection $\pi_V: Y_{\cap}\to [1,1+\epsilon]_r\times \Gamma$ induces an isomorphism $\rd \pi_V|_{\xi_J}:\xi_J\to T ([1,1+\epsilon]_r\times \Gamma)$, which is identity on $\xi_{\Gamma}$ and sends $\partial_r+f'(r)\partial_p$ to $\partial_r$ and $R_{\Gamma}-v(r)\partial_{\theta} \to R_{\Gamma}$. The metric $g_{\xi_J}=\rd \alpha(\cdot, J\cdot)$ on $\xi_J$ will be pushed forward to the metric such that 
$$|X+a\partial_r+bR_{\Gamma}|^2=r\rd \alpha_{\Gamma}(X, J_{\Gamma}X)+\left(\frac{g(r)}{r}-\frac{f'(r)g(r)v(r)}{r}\right)a^2+\left(\frac{r}{g(r)}-\frac{f'(r)v(r)r}{g(r)}\right)b^2$$
As $f'(r)<0$, $\displaystyle \lim_{r\to 1+\epsilon} f'(r)=-\infty$, $v(r)>0$ on $[1,1+\epsilon]$, $g(r)\ge 1$, and $g(r)=-\frac{f'(r)}{2}$ near $1+\epsilon$,  there exists $C>0$ such that  
$$ \frac{g(r)}{r}-\frac{f'(r)g(r)v(r)}{r}\ge C, \quad \frac{r}{g(r)}-\frac{f'(r)v(r)r}{g(r)} \ge C.$$ As a consequence, we have 
$$\mathrm{vol}_{g}(\pi_V\circ u|_{u^{-1}(\widehat{Y}_\cap)})\le  \int_{u^{-1}(\widehat{Y}_\cap)} u^*\rd \alpha,$$
where the metric $g$ is defined by 
$$|X+a\partial_r+bR_{\Gamma}|^2=r\rd \alpha_{\Gamma}(X, J_{\Gamma}X)+Ca^2+Cb^2.$$

As $u$ is completely contained in $\widehat{Y}_V\cup \widehat{Y}_{\cap}$, the above discussion shows that  $\pi_V\circ u$ will have a sufficiently small area w.r.t.\ to a fixed metric on $V$ for $f-1$ sufficiently $C^2$-small. Then by \Cref{lemma:non-trivial}, we have $\pi_V\circ u$ must be trivial in rational homology if $f-1$ is sufficiently small, as $\int u^*\rd \alpha$ is sufficiently small. The cases for \eqref{eqn:S_1} and \eqref{eqn:phi-1} are similar, as $X_H$ is parallel to $R$
\end{proof}

\begin{remark}
    Results similar to \eqref{prop:trivial_topology} were obtained by compactness arguments in \cite[Proposition 3.1]{Z23}. The above argument provides an alternative proof without appealing to the Morse-Bott compactness in \cite{BODuke}. In the case of \eqref{eqn:LCH} using the almost complex structure $J^F$, the same conclusion can be drawn from the foliation in \Cref{prop:foliation_symp} similar to \Cref{cor:disk}.
\end{remark}

We use $\cM_J(\gamma_{p_+},\gamma_{p_-})$ to denote the moduli space in \eqref{eqn:LCH}. By \Cref{prop:trivial_topology}, we have $u\in \cM_J(\gamma_{p_+},\gamma_{p_-})$ have trivial homology in the $V$-direction. In particular, we have the virtual dimension of $\cM_J(\gamma_{p_+},\gamma_{p_-})$ is $\ind(p_+)-\ind(p_-)-1$.

\begin{proposition}\label{prop:curve}
    Suppose $f-1$ in \S \ref{s3} is sufficiently $C^2$ small. For any almost complex structure on $\widehat{V\times \disk}$, which on the positive end is a generic perturbation of those considered in \Cref{prop:trivial_topology}, we have $q_{\gamma_{\max}}$ is closed and gives a class in $\LCH^{<2\pi+\delta}_{2n-3-*}(Y_f,\epsilon_{V\times \disk})$. $\cM_J(\gamma_{p_{\max}},q)$ is a smooth oriented manifold of dimension $0$ and the signed count of $\cM_J(\gamma_{p_{\max}},q)$ is $1$ for generic $q \in V\times \disk$. $\cM_J(\gamma_{p_{\max}},\gamma_{p})$ is compact and cut out transversely whenever the virtual dimension is zero and $\#\cM_J(\gamma_{p_{\max}},\gamma_{p})=0$.
\end{proposition}
\begin{proof}
We follow similar arguments in \cite{Z23}, but we need to keep track of homology classes of curves for the purpose in this paper. Similar to \cite[\S 2.4, 2.6]{Z23}, by consider a product Hamiltonian on $\widehat{V}\times \C$, we can compute $SH^{*,<\delta,<2\pi+\delta}_{+,\SS^1}(V\times \disk)\stackrel{\tau}{\simeq} H^*(V\times \disk)[1]$. Combining \cite[Proposition 2.9]{Z23} with \Cref{thm:BO}, we get the following diagram:
$$
\xymatrix{ 
SH^{*,<\delta,<2\pi+\delta}_{+,S^1}(V\times \disk) \ar[r] \ar[rrd]^{\simeq} & SH^{*,<2\pi+3\delta}_{+,\SS^1}(V\times \disk) \ar[r]^{\simeq} \ar[rd] & \LCH^{<2\pi+3\delta}_{2n-3-*}(Y,\epsilon_{V\times \disk})\ar[d]\\
& & H^{*+1}(V\times \disk)
} $$
That is $\LCH^{<2\pi+3\delta}_{2n-3-*}(Y_f,\epsilon_{V\times \disk})= \LCH^{<2\pi+\delta}_{2n-3-*}(Y_f,\epsilon_{V\times \disk}) \stackrel{\tau}{\to} (Y,\epsilon_{V\times \disk})$ is surjective. Now if we use the almost complex structure $J^F$ in \S \ref{ss:foliation} to set up $\LCH^{<2\pi+\delta}_{2n-3-*}(Y_f,\epsilon_{V\times \disk})$, we know that any curve $u$ contributing to $\tau$ must have trivial homology in the $V$-direction. Combined with \Cref{prop:trivial_topology}, we know the broken curves in the compactification also have this property. Even though the almost complex structure $J^F$ in \S \ref{ss:foliation} can not guarantee transversality, the above discussion implies that a small perturbation to $J^F$ to achieve transversality will maintain the property that curves contributing to $\delta$ must have trivial homology in the $V$-direction. Then we can get their virtual dimension using the Morse index of $p$, in particular, the only possibility to contribute to $1\in H^*(V\times \disk)$ is from $\gamma_{p_{\max}}$. As there is only one local maximum and the differential on $\LCH$ respective the Morse index by \Cref{prop:trivial_topology}, we know that $\gamma_{p_{\max}}$ is closed and is mapped to $\pm 1\in H^*(V\times \disk)$. For simplicity, we fix orientation for the orientation line associated with $\gamma_{p_{\max}}$, such that $\tau(q_{\gamma_{p_{\max}}})=1$. As the $1$ is represented by the stable manifold of a unique local minimum point $q$ of an axillary Morse function on $V\times \disk$, $\tau(q_{\gamma_{p_{\max}}})=1$ is translated to $\#\cM_J(\gamma_{p_{\max}},q)=1$. 

So far, we have established the claim for $J$ close to $J^F$ in \S \ref{ss:foliation}. To prove the claim for a compatible $J$, for example compatible $J$ considered in \Cref{prop:trivial_topology}, using the interpolation in \Cref{prop:interpolation}, we can build an almost complex structure on $\widehat{V\times \disk}$ which is compatible on the positive end and is $J^F$ on a neighborhood of $V\times \disk$, such that if we apply neck-stretching along $Y_f$, we get the trivial cobordism described after \Cref{prop:interpolation}, $(\widehat{Y}_f,J^F)$ and $(\widehat{V\times \disk},J^F)$. \Cref{prop:trivial_topology} applies to this trivial cobordism as well for $f-1$ sufficiently small by the same proof. Therefore, for a sufficiently stretched almost complex structure, we must have the curve contributing to $1\in H^*(V\times \disk)$ has trivial homology in $V$. Now the almost complex structure, denoted by $J_0$,  near the positive end is compatible. Then we can consider the family of moduli spaces $\cM_{J_t}(\gamma_{p_{\max}},q)$ and $\cM_{J_t}(\gamma_{p_{\max}},\gamma_{p})$ for a family of almost complex structures connecting $J_0$, which is close to $J_F$ in the bounded domain, to another almost complex structure $J_1$, such that $J_0=J_1$ in the positive end. The boundaries involve the moduli spaces for $t=0,1$, and those involves $\cup_t \cM_{J_{t}}(\gamma_{p_{\max}},\gamma_p)$ with virtual dimension $0$. As \Cref{prop:trivial_topology} holds for this family  $J_t$ is the same on the symplectization, the virtual dimension of $\cup_t \cM_{J_{t}}(\gamma_{p_{\max}},\gamma_p)$ is at least $1$, since $p_{\max}$ is unique local maximum. Therefore, the claim follows.
\end{proof} 

\begin{proposition}\label{prop:curve2}
Let $\nu$ be an embedded loop in $V\times \disk$. For any almost complex structure on $\widehat{V\times \disk}$, which on the positive end is a generic perturbation of those considered in \Cref{prop:trivial_topology}, the compactification $\overline{\cM}_J(\gamma_{p_{\max}},\nu)$  of $\cM_J(\gamma_{p_{\max}},\nu)$ is a smooth manifold with boundary of dimension $1$, 
One can pair up the boundary points of $\overline{\cM}_J(\gamma_{p_{\max}},\nu)$, such that $\mathrm{ev}(u)=u(0)$ is well-defined on the quotient $\overline{\cM}_J(\gamma_{p_{\max}},\nu)/\sim$ and $\mathrm{ev}$ is of degree one to $\nu$.  
\end{proposition}
\begin{proof}
We choose a generic $J$ (using generic perturbations at the positive end near Reeb orbits is sufficient), such that $\cM_J(\gamma_p,\nu)$ and $\cM_J(\gamma_p,\gamma_{p'})$ are cut out transversely for all critical points $p,p'$. Then the compactification of $\overline{\cM}_J(\gamma_{p_{\max}},\nu)$ is $\cM_J (\gamma_{p_{\max}},\nu) \bigcup \cM_J(\gamma_{p_{\max}},\gamma_p)\times \cM_J(\gamma_p,\nu)$, such that $\dim \cM_J(\gamma_{p_{\max}},\gamma_p)=0$ and $\dim  \cM_J(\gamma_p,\nu)=0$. Those $\cM_J(\gamma_{p_{\max}},\gamma_p)$ and $\cM_J(\gamma_p,\nu)$ with expected dimension $0$ are compact. By \Cref{prop:curve}, we know that $\#\cM_J(\gamma_{p_{\max}},\gamma_p)=0$. In particular, there is a (non-canonical) pairing in $\cM_J(\gamma_{p_{\max}},\gamma_p)$ by points with opposite orientations. As a consequence, the quotient $\overline{\cM}_J(\gamma_{p_{\max}},\nu)/\sim$ is a compact manifold without boundary, and the evaluation map to $\nu$ is well-defined and continuous. For a fixed point $q$ on the image of $\nu$, we may assume the almost complex structure $J$ is generic such that $\cM_J(\gamma_{p_{\max}},q)$ is the compactification and regular as in \Cref{prop:curve}. As a consequence, $q$ does not meet $\cM_J(\gamma_p,\nu)$ by the evaluation map for those $\cM_J(\gamma_{p_{\max}},\gamma_p)\ne \emptyset$, for otherwise, $\cM_J(\gamma_{p_{\max}},q)$ is proper subset of its compactification. As a consequence, $\mathrm{ev}^{-1}(q)|_{\overline{\cM}_J(\gamma_{p_{\max}},\nu)/\sim}=\mathrm{ev}^{-1}(q)|_{\cM_J(\gamma_{p_{\max}},\nu)}\simeq \cM_J(\gamma_{p_{\max}},q)$ and transversality for  $\cM_J(\gamma_{p_{\max}},q)$ implies that $q$ is a regular value of $\mathrm{ev}$. Therefore, by \Cref{prop:curve}, we know that the degree of $\mathrm{ev}$ is $1$.
\end{proof}

Note that for symplectic filling $W$ of $H\subset Y_f$, by requiring $\partial U\subset \partial V$ is a contact submanifold such that the Reeb vector on $\partial V$ restricts to Reeb vector field on $\partial U$, there exist compatible almost complex structures on $\widehat{V\times \disk}$ such that $\widehat{W}$ is such that \Cref{prop:curve,prop:curve2} hold for $p$ and $\nu$ do not intersect $\widehat{W}$, as we can perturb $J$ near $p$ and $\nu$, which are outside of $\widehat{W}$. In the following section, we will suppress the requirement on almost complex structures for simplicity. 
\section{Homotopy type of the complement}\label{s5}
We now proceed to prove \Cref{thm:V}.
\begin{proposition}\label{prop:homology}
    Under the assumptions of \Cref{thm:V}, let $W$ be a symplectic filling of $H=\partial(U\times \disk)$ in $V\times \disk$, then $V\times \{(0,1)\}\backslash U\times \{(0,1)\}\to V\times \disk \backslash W$ induces an isomorphism on homology.
\end{proposition}
\begin{proof}
    By \cite[Theorem 1.1]{Z23} and universal coefficient theorem, we have $U\times \{(0,1)\}\to W$ induces an isomorphism on homology. Then by the long exact sequence of pairs and the five lemma, we have $H_*(V\times \{(0,1)\}, U\times \{(0,1)\})\to H_*(V\times \disk, W)$ is an isomorphism. Let $S$ be the circle bundle over $W$ from the normal bundle of $W$ in $V\times \disk$ and $P$ be the restriction of $S$ to $U\times \{(0,1)\}$. By the Gysin exact sequence, $P\subset S$ also induces an isomorphism on homology. By excision, we have a long exact sequence
    $$\ldots \to H_*(P)\to H_*(V\times \{(0,1)\}\backslash U\times \{(0,1)\}) \to H_*(V\times \{(0,1)\}, U\times \{(0,1)\}) \to  \ldots $$
    and similarly for $H_*(V\times \disk, W)$. Therefore the five lemma implies that $V\times \{(0,1)\}\backslash U\times \{(0,1)\}\to V\times \disk \backslash W$ induces an isomorphism on homology.
\end{proof}

\begin{remark}
    When $V=\disk^n$ and $U=\disk^{n-1}$. By Eliashberg-Floer-McDuff theorem \cite{McDuff91}, $W$ is diffeomorphic to a $2n$-dimensional ball. The one pointed compactification of the interior of $(\disk^{n+1},W)$ is a pair $(\overline{\disk^{n+1}}=\Sphere^{2n+2},\overline{W}=\Sphere^{2n})$. Then by Alexander duality, we have 
	$$H_*(\disk^{n+1}\backslash W)=H_*(\Sphere^{2n+2}\backslash \overline{W})=H_*(\Sphere^1).$$
	The first homology is generated by loops around $\overline{W}$, and an isomorphism to $\Z$ can be represented by the linking number. Therefore $\Sphere^{2n+1}\backslash \Sphere^{2n-1}\to \disk^{n+1}\backslash W$ induces an isomorphism on homology.
\end{remark}

\begin{proposition}\label{prop:homotopy}
    Under the assumptions of \Cref{thm:V}, let $W$ be a symplectic filling of $H=\partial(U\times \disk)$ in $V\times \disk$, then $V\times \{(0,1)\}\backslash U\times \{(0,1)\}\to V\times \disk \backslash W$ induces a surjective map on the fundamental group.
\end{proposition}
\begin{proof}
  By van-Kampen theorem, $V\times \{(0,1)\}\backslash U\times \{(0,1)\} \to Y\backslash H$ induces an isomorphism on the fundamental group. For a fixed point $q$ in $V\times \disk$ outside $W$, we assume $\cM_J(\gamma_{p_{\max}},q)$ is regular. Take $[u_0]\in \cM_J(\gamma_{p_{\max}},q)$ and $\nu$ is an embedded loop in $V\times \disk \backslash W$ based at $q$. Let $\cM\subset \overline{\cM}_J(\gamma_{p_{\max}},\nu)/\sim$ be the component containing $[u_0]$ and the evaluation map to $\nu$ is degree $1$ by \Cref{prop:curve2}. By \Cref{prop:intersection}, curves in $\cM$ do not intersect $\widehat{W}$ as they meet points outside of $\widehat{W}$. We fix a parametrization of the simple Reeb orbit $\gamma_{p_{\max}}$. For $[u]\in \cM\cap \cM_{J}(\gamma_{p_{\max}},\nu)$, we can fix a holomorphic map $u:\C\to \widehat{V\times \disk}$ such that $\lim_{x\to \infty} u(x+\mathbf{i}0)=\gamma_{p_{\max}}(0)$ and $\int_{\disk} |\rd u|^2=1$, The first condition fixes the $\Sphere^1$-factor of $\Aut(\C,0)$. The second condition fixes the $\R$-factor of $\Aut(\C,0)$, as $\int_{\disk(r)} |\rd u|^2$ is a strictly increasing function w.r.t.\ the radius $r$, for otherwise, we get a contradiction with the unique continuation of holomorphic curves. When such $u\in \cM\cap \cM_{J}(\gamma_{p_{\max}},\nu)$ converging to a glued point in $\cM$, the $C^{\infty}_{loc}$ limit of them gives rise to a curve (as maps modulo reparameterization) in $\cM_J(\gamma_p,\nu)$ as we require that $u(0)\in \ima \nu$. However, as parametrized maps, the $C^{\infty}_{loc}$ limits from two sides of a glued point may be different. Therefore we have $k$ intervals $\{I_i=(L_i,R_i)\}$ corresponding to the components of $\cM\cap \cM_{J}(\gamma_{p_{\max}},\nu)$ and each $I_i$ parametrizes a continuous family of maps $u_i(t)$ from $\C$ to $\widehat{V\times \disk}$ in the $C^{\infty}_{loc}$ topology, such that the limit curve $[u_{i}(R_i)]=[u_{i+1}(L_{i+1})]$ as unparameterized curves where $u_{k+1}(L_{k+1})=u_1(L_1)$.  We can write $u_{i}(R_i)=u_{i+1}(L_{i+1})\circ \phi_i$ for $\phi_i\in \Aut(\C,0)$. We define $\nu_s$ by concatenating alternatively $u_i(t)(s)$ with $u_i(R_i)(p_{i,s})$, where $p_{i,s}$ is a path in $\C$ connecting $s$ and $\phi^{-1}_i(s)$ depending continuously on $s$, such that $\lim_{s\to \infty } |p_{i,s}|=\infty$. We can have such a path as $|\phi_i^{-1}(s)|\to \infty$ when $s\to \infty$ for $\phi\in \Aut(\C,0)$. Then $\nu_0$ is a loop homotopic to $\nu$ based at $q$, and $\nu_s$ is a homotopy in $\widehat{V\times \disk}\backslash \widehat{W}$, such that for $s\gg 0$, $\nu_{s}$ is contained in $ \R_+\times (Y\backslash H)\subset \widehat{V\times \disk}\backslash \widehat{W}$. The base point moves on the fixed path $u_0(s)$. This shows that $ \R_+\times (Y\backslash H)\subset \widehat{V\times \disk}\backslash \widehat{W}$ induces a surjective map of the fundamental group. 
\end{proof}

\begin{remark}\label{rmk:general}
    \Cref{prop:homotopy} is local in the following sense. If we apply subcritical/flexible surgery to $\partial(V\times \disk)\backslash \partial(U\times \disk)$, all the ingredients in the proof should persist. \Cref{prop:intersection} clearly holds. That $\gamma_{p_{\max}}$ presents a closed class in linearized contact homology and is mapped to $1\in H^0(V\times \disk)$ should follow from the surgery formulae in \cite{BourEkEl} with exceptions built from one case: if we apply connect sum with the standard sphere $(\Sphere^{2n+1},\xi_{std})$, i.e.\ applying a combination of a $0$-surgery and a $1$-surgery to the complement. In this case, $\gamma_{p_{\max}}$ is not closed on the chain level w.r.t.\ the contact form following \cite{BourEkEl,Laz20,Yau} after surgery. However, such surgeries do not change the contact topology, i.e.\ we can isotope it back to \Cref{prop:homotopy}. To rigorously prove those claims is non-trivial. However, the following special case is easier to establish: \Cref{prop:homotopy} holds if we apply $2k$-subcritical surgeries to  $\partial(V\times \disk)\backslash \partial(U\times \disk)$. Following \cite[Proposition 4.4]{Laz20}, we get extra Reeb orbits from the surgery handles whose $\Z/2$-grading is the same as $\gamma_{p_{\max}}$. Note that $\gamma_{p_{\max}}$ is simple; we can argue that curves with the positive puncture asymptotic to $\gamma_{p_{\max}}$ are cut out transversely. Therefore, the key properties of $\gamma_{p_{\max}}$ are preserved under those surgeries, and the proof goes through.
\end{remark}

\begin{proof}[Proof of \Cref{thm:V}]
    By \Cref{prop:homology} and $\pi_1(V\backslash U)$ is abelian, we have the following commutative diagram
    $$
    \xymatrix{
     \pi_1(V\times \{(0,1)\}\backslash U\times \{(0,1)\}) \ar[r]\ar[d]^{\simeq} & \pi_1(V\times \disk \backslash W)\ar[d] \\
     H_1(V\times \{(0,1)\}\backslash U\times \{(0,1)\}) \ar[r]^{\qquad \qquad \simeq} & H_1(V \times \disk \backslash W) 
    }
    $$
    By \Cref{prop:homotopy}, we have $\pi_1(V\times \{(0,1)\}\backslash U\times \{(0,1)\}) \to \pi_1(V\times \disk \backslash W)$ is an isomorphism. Combining \Cref{prop:homology} with Hurewicz theorem and Whitehead theorem, $V\times \{(0,1)\}\backslash U\times \{(0,1)\}\hookrightarrow V\times \disk \backslash W$ is a homotopy equivalence. 
\end{proof}

\begin{proof}[Proof of \Cref{thm:main}]
By \Cref{thm:V}, $\disk^{n+1}\backslash W$ is homotopy equivalent to $\Sphere^1$. The claim follows from Levine's theorem \cite{Levine65}.    
\end{proof}
\section{Filling of the binding of a trivial open book}\label{s6}
\subsection{Pseudocycles and integral homology}
We will use pseudocycles to represent the integer homology classes as in \cite{Pseudocycles}. Let $X$ be a smooth manifold. A subset $Z$ of $X$ is said to have dimension at most $k$ if there exists a $k$-dimensional manifold $Y$ and a smooth map $h: Y\to X$ such that the image of $h$ contains $Z$. If $\rho:M\to X$ is a continuous map between topological spaces, the boundary of $\rho$ is the set
$$\Bd\rho = \bigcap_{K\subset M \text{compact}}\overline{f(M-K)}.$$
A smooth map $\rho:M\to X$ is a $k$-pseudocycle if $M$ is an oriented $k$-manifold, $f(M)$ is precompact in $X$ and the dimension of $\Bd\rho$ is at most $k-2$. Two $k$-pseudocycles $\rho_0:M_0\to X,\rho_1:M_1\to X$ are equivalent if there exists a smooth oriented manifold $\widetilde{M}$ and a smooth map $\widetilde{\rho}:\widetilde{M}\to X$ such that the image of $\widetilde{\rho}$ is a precompact subset of $X$ and 
$$\dim \Bd \widetilde{\rho}\le k-1, \quad \partial \widetilde{M}=M_1-M_0, \quad  \widetilde{\rho}|_{M_0}=\rho_0, \quad \widetilde{\rho}|_{M_1}=\rho_1.$$
We denote the set of equivalence classes of pseudocycles in $X$ by $\mathcal{H}_*(X)$, which is graded by $\N$.
\begin{theorem}[{\cite[Theorem 1.1]{Pseudocycles}}]\label{thm:Zinger}
    $H_*(X)$ is naturally isomorphic to $\mathcal{H}_*(X)$.
\end{theorem}
The isomorphism $\mathcal{H}_*(X)\to H_*(X)$ is defined as follows: Let $\rho:M\to X$ be a $k$-pseudocycle, one can choose a compact $k$-submanifold with boundary, $\overline{N}\subset M$, such that $\rho (M\backslash N)$ is contained in a small neighborhood $U$ of $\Bd \rho$ where $H_k(X)\simeq H_k(X,U)$. Then $\rho_*[\overline{N}]\in H_k(X,U)\simeq H_k(X)$ is the image of $\rho$ in $H_k(X)$, where $[\overline{N}]\in H_k(\overline{N},\overline{N}\backslash N)$ is the fundamental class. One way to represent all classes in $H_*(X)$ by pseudocycles is as follows: we consider a triangulation of $X$, for a closed simplicial chain, we can obtain an open smooth manifold by taking the interior of the simplexes and the interior of the faces which are glued as they can be canceled by the closeness. The natural map from this open manifold to $X$ induces a pseudocycle, as the boundary is precisely the codimension $2$ skeleton of the simplexes involved. This pseudocycle represents the same homology class as the simplicial chain by the discussion above.
\begin{remark}
    By \cite[Remark 2]{Pseudocycles}, the theory of pseudocycles works for continuous maps. The only place needed smoothness ($C^1$ is sufficient) is the definition of dimension, so that \cite[Proposition 2.2]{Pseudocycles} holds\footnote{For otherwise, we have pathological examples like the Peano curve.}. 
\end{remark}

\subsection{Probing topology by holomorphic spheres}
To prove \Cref{thm:binding}, we use holomorphic spheres to probe the topology of the filling, i.e.\ a combination of arguments in \cite{BGZ} with classical intersection theories\footnote{It should be possible to prove \Cref{thm:V} using holomorphic spheres as well. But the method using the intersection theory of punctured holomorphic curves is more flexible and local, in particular, it is unclear whether \Cref{rmk:general} can be established using holomorphic spheres.}. We consider the moduli space
$$ \cN:=\left\{ u:\CP^1\to \widehat{V}\times \CP^1\left|\begin{array}{c}\overline{\partial}_J u=0, [u]=[\{\ast\} \times \CP^1],  u(0)\in \widehat{V}\times \disk^{\circ} \\ u(1)\in \widehat{V}\times \{(2,0)\}, u(\infty)\in \widehat{V}\times \{\infty\} \end{array}\right.\right\}$$
Here we use an almost complex structure such that 
\begin{itemize}
    \item It is compatible with $\widehat{\lambda}_V\oplus \omega_{\CP^1}$;
    \item It is a product complex structures on $(\widehat{V}\backslash V) \times \CP^1$;
    \item The $r$-coordinate on the collar end $(\widehat{V}\backslash V)=\partial V \times (1,\infty)_r$ is plurisubharmonic;
    \item $\widehat{W}$ is a holomorphic hypersurface;
    \item $\widehat{V}\times \{(2,0)\}$ and $\widehat{V}\times \{\infty\}$ are holomorphic hypersurface.
\end{itemize}
\begin{proposition}\label{prop:proper}
    $\mathrm{ev}_0:\cN\to \widehat{V}\times \disk^{\circ}, u\mapsto u(0)$ is a proper map of degree $1$.
\end{proposition}
\begin{proof}
    Since the homology class $[\{\ast\} \times \CP^1]$ can not decompose into multiple homology classes with positive symplectic area, compactification of $\mathrm{ev}_0^{-1}(K)$ does not involve sphere bubbles for any compact subset $K\subset \widehat{V}\times \disk^{\circ}$. The degree can be computed by looking at a point $p$ outside $V\times\disk^{\circ}$ in $\widehat{V}\times \disk^\circ$. By the maximum principle for the $r$-coordinate, $\mathrm{ev}_0^{-1}(p)$ must be the fiber sphere in $\widehat{V}\times \CP^1$, which is cut out transversely. With a suitable orientation on $\cN$, the degree of $\mathrm{ev}_0$ is $1$.
\end{proof}
\begin{proposition}\label{prop:intersection2}
    Under the assumptions of \Cref{thm:binding}, any curve in $\cN$ intersects $\widehat{W}$ once. 
\end{proposition}
\begin{proof}
The intersection number only depends on the homology classes. 
By pushing the curve to the end of $\widehat{V}$, it is clear the intersection number is $1$.
\end{proof}
Let $\rho:M\to \widehat{V}\times \disk^{\circ}$ be a smooth map from a manifold $M$, we define 
$\cN(\rho):=\{ u:\CP^1\to \widehat{V}\times \CP^1,m\in M|\overline{\partial}_J u=0, [u]=[\{\ast\} \times \CP^1],  u(0)=\rho(m), u(1)\in \widehat{V}\times \{(2,0)\}, u(\infty)\in \widehat{V}\times \{\infty\} \}$.

\begin{proposition}\label{prop:pseudocycle}
    Let $\rho:M\to \widehat{V}\times \CP^1\backslash\{(2,0),\infty\}$ be a $k$-pseudocycle from a triangulation of $\widehat{V}\times \CP^1\backslash\{(2,0),\infty\}$ explained after \Cref{thm:Zinger}. For generic $J$, the evaluation map $\mathrm{ev}_0:\cN_J(\rho)\to \widehat{V}\times \disk, (u,m)\mapsto u(0)$ is a  $k$-pseudocycle equivalent to $\rho$.
\end{proposition}
\begin{proof}
As $\rho$ is from a triangulation of $\widehat{V}\times \CP^1\backslash\{(2,0),\infty\}$, there exists a possibly disconnected manifold $Y$ whose components have dimension at most $k-2$ and a smooth map $h:Y\to \widehat{V}\times \CP^1\backslash\{(2,0),\infty\}$, such that $h(Y)$ contains the boundary $\Bd \rho$, i.e.\ $Y$ is the disjoint union of the interior of the simplexes of the codimension $2$ skeleton of the simplexes in defining $\rho$. For a generic almost complex structure, $\cN_J(h)$ is an at most $k-2$ dimensional manifold and $\cN_J(\rho)$ is a $k$-dimensional manifold. It is clear that $\Bd(\mathrm{ev}_0|_{\cN_J(\rho)})\subset \mathrm{ev}_0(\cN_J(h))$. Let $N$ be an open submanifold of $M$, such that the pseudocycle represent $\rho_*[\overline{N}]\in H_k(\widehat{V}\times \CP^1\backslash\{(2,0),\infty\}, U)\simeq H_k(\widehat{V}\times \CP^1\backslash\{(2,0),\infty\})$, where $U$ is a small neighborhood of $h(Y)$. Since $\pi:\cN_J(\rho|_{\overline{N}})=\mathrm{ev}_0^{-1}(\overline{N})\to \overline{N}, (u,m)\mapsto m$ is of degree $1$ by \Cref{prop:proper}. Hence $\rho$ and $\mathrm{ev}_0$ induce the same homology class.
\end{proof}

\begin{proposition}\label{prop:surj}
    Under the assumption of \Cref{thm:binding}, $W\subset V\times \disk$ induces a surjection on homology.
\end{proposition}
\begin{proof}
    Let $\rho:M\to V\times \disk$ be a $k$-pseudocycle. We may assume $\overline{\rho(M)}$ does not interest $W$ as we can represent all the homology classes by pseudocycles close to the boundary part $V\times S^1$. For every holomorphic disk $u$ in $\cN_J(\rho)$, $u$ intersects $\widehat{W}$ once transversely by \Cref{prop:intersection2} at $p_u\in \CP^1\backslash\{0,1,\infty\}$.  Then we have $$\mathrm{ev}^{[0,1]}:[0,1]_t\times \cN_J(\rho) \to \widehat{V} \times \mathbb{CP}^1\backslash\{\infty\}, \qquad (t,[(u,m)]) \mapsto u(tp_u) $$
    defines a pseudocycle equivalence between $\mathrm{ev}^{t=0}$ and $\mathrm{ev}^{t=1}$. By \Cref{prop:pseudocycle}, $\mathrm{ev}^{t=0}$ is equivalent to $\rho$. Since $\mathrm{ev}^{t=1}$ is contained in $\widehat{W}$. The claim follows.  
\end{proof}

\begin{proposition}
    Under the assumption of \Cref{thm:binding}, $W\subset V\times \disk$ induces a surjection on the fundamental group.
\end{proposition}
\begin{proof}
    We fix a point $q$ near the boundary of $\partial(V\times \disk)$ and contained in $W$. If we view $q$ as a map from $\{\ast\}$ to $V\times \disk$, $\cN_J(q)$ consists of a single curve (the fiber sphere) as $q$ is close to $\partial V \times \{0\}$. Let $\gamma$ be an embedded loop in $V\times \disk$ based at $q$. Then $\cN_J(\gamma)$ is a one-dimensional closed manifold for a generic choice of $J$. By \Cref{prop:intersection2}, each $u\in \cN_J(\gamma)$ intersects $\widehat{W}$ at $p_u\in \CP^1\backslash\{1,\infty\}$. Let $\cM\simeq \Sphere^1$ be the connected component of $\cN_J(\gamma)$, such that $\mathrm{ev}_0:\cM \to \gamma$ is of degree $1$. In particular, $\cN_J(q)$ is contained in $\cM$. 
    Then 
    $$[0,1]_t\times \cM \to \widehat{V}\times \C, (t,u)\mapsto u(tp_u)$$
    is a homotopy from $\mathrm{ev}_0:\cM \to \gamma\to \widehat{V}\times \C$ to a loop in $\widehat{W}$ based at $q$. That is 
    $W\subset V\times \disk$ induces a surjection on the fundamental group.
\end{proof}

\begin{proposition}
    Under the assumption of \Cref{thm:binding}, $W\subset V\times \disk$ induces an injection on homology.
\end{proposition}
\begin{proof}
    Let $\rho:M\to W$ be a $k$-pseudocycle and $\widetilde{\rho}:\widetilde{M}\to V\times \disk$ is an equivalence from $\rho$ to $\emptyset$. We may assume both are from a triangulation and simplicial homology. Then $\mathrm{ev}_0:\cN_J(\widetilde{\rho})\to V\times \disk$ is an equivalence from $\mathrm{ev}_0:\cN_J(\rho)\to W$ to $\emptyset$. For curves $u\in \cN_J(\widetilde{\rho})$, there is a unique point $p_u\in \CP^1\backslash \{1,\infty\}$ such that $u(p_u)\in \widehat{W}$ by \Cref{prop:intersection2}. $p_u$ is $0$ if $u$ is from $\cN_J(\rho)$. Then $\mathrm{ev}:\cN_J(\widetilde{\rho})\to \widehat{W}, (u,\widetilde{m}) \mapsto u(p_u)$ defines an equivalence from $\rho$ to $\emptyset$ as pseudocycles in $\widehat{W}$. That is, the induced map on homology is an injection.
\end{proof}

\Cref{thm:binding} follows from the above propositions and the Whitehead theorem. 
\bibliographystyle{alpha} 
\bibliography{ref}
\Addresses
\end{document}